\documentclass[12pt,reqno,a4paper]{amsart}
\usepackage{amsmath,amssymb,amsthm,amsaddr}
\usepackage{enumitem}
\usepackage[margin=2.5cm,top=4cm,bottom=3.5cm,footskip=1cm,headsep=1.5cm]{geometry}
\usepackage[utf8]{inputenc}
\usepackage{mathtools}
\usepackage{float}
\usepackage{tikz}
\usetikzlibrary{matrix}
\usepackage{amsthm}

\usepackage{cleveref}
\usepackage{cancel}
\usepackage[british]{babel}
\usepackage[caption=false]{subfig}

\author{H. Egger \and B. Radu}
\address{Department of Mathematics, TU Darmstadt, Germany}
\email{egger@mathematik.tu-darmstadt.de}
\email{radu@gsc.tu-darmstadt.de}

\definecolor{mygray}{rgb}{.5,.5,.5}

\title[A Mixed FEM with mass-lumping for wave propagation]{A mass-lumped mixed finite element method\\ for acoustic wave propagation}

\newtheorem{lemma}{Lemma}[section]
\newtheorem{problem}[lemma]{Problem}
\newtheorem{theorem}[lemma]{Theorem}

\theoremstyle{definition}
\newtheorem{remark}[lemma]{Remark}
\newtheorem*{example*}{Example}

\def\div{\mathrm{div}}
\def\grad{\nabla}

\def\tnorm{|\!|\!|}

\def\dt{\partial_t}
\def\dtt{\partial_{tt}}
\def\dttt{\partial_{ttt}}

\def\u{u}
\def\p{p}

\def\RR{\mathbb{R}}

\def\Th{\mathcal{T}_h}

\def\R{\mathbb{R}}
\def\dtau{d_\tau}
\def\P{\text{P}}
\def\T{\mathcal{T}}
\def\BDM{\text{BDM}}

\numberwithin{equation}{section}
\numberwithin{table}{section}
\numberwithin{figure}{section}

\begin{document}


\begin{abstract}
We consider the numerical approximation of acoustic wave propagation in the time domain by a mixed finite element method based on the $\text{BDM}_1$--$\text{P}_0$ spaces. A mass-lumping strategy for the $\text{BDM}_1$ element, originally proposed by Wheeler and Yotov in the context of subsurface flow problems, is utilized to enable an efficient integration in time. By this mass-lumping strategy, the accuracy of the mixed method is formally reduced to first order. We will show, however, that the numerical approximation still carries global second order information, which is expressed as super-convergence of the numerical approximation to certain projections of the true solution. Based on this fact, we propose post-processing strategies for both variables, the pressure and the velocity, which yield piecewise linear approximations of second order accuracy. A complete convergence analysis is provided for the semi-discrete and corresponding fully-discrete approximations, which result from time discretization by the leapfrog method. In addition, some numerical tests are presented to illustrate the efficiency of the proposed approach.
\end{abstract}

\maketitle

\begin{quote}
\end{quote}

\begin{quote}
\noindent
{\small {\bf AMS-classification (2000):}
35L05, 35L50, 65L20, 65M60}
\end{quote}

\section{Introduction}

Due to the many applications in acoustics, elastodynamics or electromagnetics, the numerical approximation of wave propagation has been a topic of intensive research for many years. Various discretization schemes are available by now, including finite difference, finite volume, and finite element methods, each with their particular advantages and shortcommings. 
In this manuscript, we consider finite element methods over unstructured grids, which are very flexible concerning geometry and physical parameters, and we focus on acoustic wave propagation in the time domain as our model problem.

The numerical approximation of the second order wave equation by continuous piecewise polynomial finite elements 
is well understood; see, e.g., \cite{Baker76,Dupont73,Wheeler73a} and also \cite{BakerBramble79,DouglasDupontWheeler78} for fully discrete approximations.
Mass-lumping strategies have been proposed and analyzed in \cite{Cohen02,CohenJolyRobertsTordjman01}, in order to allow for an efficient integration in time by explicit Runge-Kutta or multistep methods.
Let us refer to \cite{Cohen02,CohenMonk98} for corresponding results in the context of electromagnetics.

For problems with anisotropic, strongly inhomogeneous, or dispersive materials, the formulation of the wave propagation problem as  hyperbolic system of first order seems to be advantageous. Corresponding mixed finite element approximations for acoustic and elastic wave propagation have been considered in \cite{CowsarDupontWheeler96,Geveci88,JenkinsRiviereWheeler02,Makridakis92}. Let us refer to  \cite{LeeMadsen90,MakridakisMonk95,Monk92} for related results in electrodynamics and to \cite{Cohen02,Joly03}
for an overview about the general approach and further references.

In this paper, we consider a mixed finite element discretization for acoustic wave propagation based on the conforming approximation for velocity and pressure in the spaces $H(\div)$ and $L^2$ with $\BDM_1$ and $\P_0$ elements, respectively; 
see \cite{BoffiEtAl08,BrezziDouglasMarini85} for details on these elements.
Together with mass-lumping, the semi-discrete approximation in space for our model problem is 
defined by the mixed variational principle
\begin{align}
(\dt u_h(t),v_h)_h - (p_h(t), \div \, v_h) &= 0, \label{eq:sys1}\\  
(\dt p_h(t),q_h) + (\div \, u_h(t),q_h) &= 0,   \label{eq:sys2}
\end{align}
with $u_h$ and $p_h$ denoting the numerical approximations for velocity and pressure. 
The modified scalar product $(\cdot,\cdot)_h$ for the velocity space is chosen in order to give rise to 
a block diagonal mass matrix \cite{WheelerYotov06}. The resulting semi-discrete problem can then be 
integrated efficiently in time by explicit time-stepping schemes. 
It is not difficult to see that the method outlined above is first order accurate, i.e., 
\begin{align} \label{eq:est1}
\|u(t) - u_h(t)\| + \|p(t) - p_h(t)\| \le C h,  
\end{align}
whenever the true solution $(u,p)$ is sufficiently smooth. 

As shown in \cite{EggerRadu16}, the standard Galerkin approximation, which uses the normal scalar product $(\cdot,\cdot)$ instead of the $(\cdot,\cdot)_h$ in equation \eqref{eq:sys1}, yields a second order approximation for the velocity and the projected pressure \cite{EggerRadu16}. This is no longer true for the method \eqref{eq:sys1}--\eqref{eq:sys2}, for which the approximation for the velocity is only of first order due to the perturbations introduced by the mass-lumping procedure.
Surprisingly, the pressure still shows super-convergence in numerical tests, i.e., 
\begin{align} \label{eq:est2}
\|\pi_h^0 p(t) - p_h(t)\| \le C h^2,
\end{align}
where $\pi_h^0$ denotes the $L^2$-orthogonal projection onto piecewise constant functions. 

As a first result of our paper, we give a proof of the estimate \eqref{eq:est2} and we show that 
\begin{align} \label{eq:est3}
\|u_h^*(t) - u_h(t)\| + \|p_h^*(t) - p_h(t)\| \le C h^2,  
\end{align}
i.e., both solution components converge with second order to a certain projection $(u_h^*,p_h^*)$ of the true solution $(u,p)$. 
This projection has to be chosen such that the numerical error of the mass-lumping has no influence on the discrete error components.
For the construction of these auxiliary functions, we here employ an inexact elliptic projection considered by Wheeler and Yotov \cite{WheelerYotov06} in the context of subsurface flow problems.

Based on the improved estimate \eqref{eq:est3}, we are able to construct in a post-processing step piecewise linear improved approximations $(\widetilde u_h,\widetilde p_h)$ which satisfy
\begin{align} \label{eq:est4}
\|u(t) - \widetilde u_h(t)\|  + \|p(t) - \widetilde p_h(t)\| \le C h^2.
\end{align}
This shows that the post-processing method is second order accurate.
The analysis of these post-processing strategies will be the second main contribution of our paper.

In addition, we also consider the corresponding fully-discrete scheme, which results from time-discretization of \eqref{eq:sys1}--\eqref{eq:sys2} by the leapfrog method \cite{HairerLubichWanner03,Joly03}. 
A full convergence analysis for the resulting method is presented, including the extension of the post-processing schemes to the fully-discrete setting. The piecewise linear approximations $(\widetilde u_h^{\,n},\widetilde p_h^{\,n})$ obtained 
by these post-precessing procedures satisfy
\begin{align} \label{eq:est5}
\|u(t^n) - \widetilde u_h^{\,n}\| + \|p(t^n) - \widetilde p_h^{\,n}\| \le C (h^2 + \tau^2).  
\end{align}
Here $h$ and $\tau$ describe the mesh and time step size, respectively. 
In summary, we thus obtain a second order approximation for the wave equation on unstructured meshes,
which can be computed efficiently in a fully explicit manner.

Let us emphasize that related results are well-known for mass-lumped mixed finite element approximations on structured grids \cite{Cohen02,MonkSuli94}, which can be understood as variational formulations of finite-difference time-domain methods \cite{Taflove,Yee66}. These methods are the gold standard for the simulation of wave propagation problems in industrial applications. 
Our results can thus be seen as a generalization of such methods to unstructured grids.

\bigskip

The remainder of the manuscript is organized as follows:
In Section~\ref{sec:prelim}, we summarize our notation and introduce the model problem considered in the rest of the paper.
Section~\ref{sec:semi} is devoted to the analysis of the semi-discrete scheme \eqref{eq:sys1}--\eqref{eq:sys2} with mass-lumping and to the proof of the estimates \eqref{eq:est1}--\eqref{eq:est3}. 
In Section~\ref{sec:semi-post}, we propose the post-processing strategies for pressure and velocity and establish the estimate \eqref{eq:est4}. 
Section~\ref{sec:time} deals with the time-discretization of \eqref{eq:sys1}--\eqref{eq:sys2} by the leapfrog scheme 
and we present a complete convergence analysis of the fully-discrete method. 
Section~\ref{sec:time-post} is concerned with the extension of the post-processing strategies to the fully discrete level and contains the proof of the final estimate \eqref{eq:est5}.  
In Section~\ref{sec:num}, we report about some numerical tests which illustrate our theoretical results and demonstrate the efficiency of the proposed method.

\section{Preliminaries} \label{sec:prelim}

Throughout the manuscript, let $\Omega\subseteq\R^d$, $d=2,3$ be some bounded Lipschitz domain and $T>0$ be a fixed time horizon. We consider the first order hyperbolic system
\begin{align}
\dt \u + \grad \p &= 0 \qquad \text{in } \Omega\times(0,T),\label{eq:sysm1}\\
\dt \p + \div\, \u &= 0 \qquad \text{in } \Omega\times(0,T),\label{eq:sysm2} 
\end{align}
and we call $u$ and $p$ the \emph{velocity} and \emph{pressure}, respectively. 
For ease of presentation, we assume that the pressure satisfies homogeneous boundary conditions, i.e., 
\begin{align}
\p &= 0 \qquad \text{on } \partial\Omega\times(0,T), \label{eq:sysm3} 
\end{align}
and we further require knowledge of the initial values
\begin{align}
u(0)=u_0\quad\text{ and }\quad p(0)=p_0\quad\text{ in }\Omega. \label{eq:sysm4} 
\end{align}
The simplicity of this model problem facilitates the presentation of our results. 
The proposed methods are, however, also applicable to problems with inhomogeneous or anisotropic coefficients, 
with lower order terms or right hand sides, and other types of boundary conditions. 
This will become clear from our analysis and the numerical tests.

The well-posedness of the initial boundary value problem \eqref{eq:sysm1}--\eqref{eq:sysm4} can be deduced from standard arguments of semi-group theory; see \cite{Evans98,Pazy83} for details. 
\begin{lemma}[Classical solution] $ $\\
For any $u_0 \in H(\div;\Omega)$ and $p_0 \in H_0^1(\Omega)$ there exists a unique classical solution 
\begin{align*}
(u,p)\in C^1(0,T;L^2(\Omega)^d\times L^2(\Omega)) \cap C(0,T;H(\div;\Omega)\times H^1_0(\Omega))
\end{align*}
of the system \eqref{eq:sysm1}--\eqref{eq:sysm4} and its norm can be bounded by the norm of the data. 
\end{lemma}
The spaces $L^2(\Omega)$, $H^1(\Omega)$, and $H(\div;\Omega)=\{u\in L^2(\Omega)^d\,|\, \div\, u \in L^2(\Omega)\}$ denote the usual Lebesgue and Sobolev spaces, and functions in $H_0^1(\Omega)$ additionally have zero trace on the boundary. 
The spaces $C^k([0,T];X)$ and $W^{k,p}(0,T;X)$ consist of functions of time with values in a Hilbert space $X$; see \cite{Evans98} for details and further notation.
By standard arguments, one obtains the following characterization of classical solutions.
\begin{lemma}[Variational characterization] $ $\\
Let $(u,p)$ denote a classical solution of \eqref{eq:sysm1}--\eqref{eq:sysm4}. Then for all $0\le t\le T$ there holds
\begin{alignat}{5}
(\dt u(t),v)_\Omega-(p(t),\div \,v)_\Omega &=0 \qquad \forall v\in H(\div;\Omega), \label{eq:var1} \\ 
(\dt p(t),q)_\Omega+(\div\, u(t),q)_\Omega &=0 \qquad \forall q\in L^2(\Omega), \label{eq:var2} 
\end{alignat}
where $(\cdot,\cdot)_\Omega$ denotes the standard scalar product of $L^2(\Omega)$.
\end{lemma}
This weak characterization will be the starting point for all our further considerations.

\section{Semi-discretization} \label{sec:semi}
In the sequel, we always assume that $\Omega$ is polyhedral in the sequel and denote by $\Th = \{K\}$ a geometrically conforming 
simplicial mesh of this domain \cite{ErnGuermond}.
As finite dimensional approximations for the function spaces $L^2(\Omega)$ and $H(\div;\Omega)$, we utilize
\begin{align}
V_h = \P_1(\Th)^d \cap H(\div;\Omega)
\qquad \text{and} \qquad
Q_h=\P_0(\T_h).
\end{align}
Here $\P_k(\Th) = \{q \in L^2(\Omega) : q|_K \in \P_k(K)\}$ denotes the space of piecewise polynomials of degree $\le k$ over the mesh $\T_h$. 
Let us note that $V_h$ amounts to the lowest order $H(\div)$-conforming $\BDM$-space \cite{BoffiBrezziFortin13,BrezziDouglasMarini85}. 
In particular, functions in $v_h \in V_h$ have continuous normal components across element interfaces, while functions $q_h \in Q_h$ are, in general, completely discontinuous across interfaces. 

In our discretization schemes, we will make use of the modified scalar product
\begin{align*}
(u_h,v_h)_{h,\Omega} = \sum\nolimits_K \int_K I_K( u_h \cdot v_h) dx,
\end{align*}
where $I_K : C(K) \to \P_1(K)$ denotes the standard nodal interpolation operator on the element $K$. 
Thus $(u_h,v_h)_{h,\Omega}$ simply amounts to integrating $(u_h,v_h)_\Omega$ inexactly on every element by the vertex rule, which can be interpreted as a mass-lumping strategy. By elementary computations \cite{WheelerYotov06}, one can verify that $(\cdot,\cdot)_{h,\Omega}$ induces a norm $\|\cdot\|_{h,\Omega}$ on $V_h$, which is equivalent to the standard norm $\|\cdot\|_{L^2(\Omega)}$, i.e., 
\begin{align} \label{eq:norm}
\frac{1}{2}\|v_h\|_{h,\Omega}\leq \|v_h\|_{L^2(\Omega)} \leq \|v_h\|_{h,\Omega}, 
\qquad \forall v_h \in V_h.
\end{align}
For the semi-discretization of problem \eqref{eq:sysm1}--\eqref{eq:sysm4}, 
we then consider the following method.
\begin{problem}[Galerkin semi-discretization]$ $\label{problem:semidiscrete}\\
Given $u_{h,0} \in V_h$ and $p_{h,0} \in Q_h$, 
find $(u_h,p_h) \in C^1(0,T;V_h \times Q_h)$ with
\begin{align} \label{eq:sysvd1}
u_{h}(0)=u_{0,h}\quad\text{ and }\quad p_{h}(0)=p_{0,h},  
\end{align}
and such that for all $0 \le t \le T$ there holds
\begin{alignat}{2}
(\dt u_h(t),v_h)_{h,\Omega}-(p_h(t),\div\, v_h)_\Omega &= 0 \qquad \forall v_h\in V_h, \label{eq:sysvd2} \\
(\dt p_h(t),q_h)_\Omega+(\div\, u_h(t),q_h)_\Omega &= 0 \qquad \forall q_h\in Q_h. \label{eq:sysvd3}
\end{alignat}
If not stated otherwise, the initial values will be chosen according to \eqref{eq:initial} below.
\end{problem}
\begin{remark} \rm 
By choosing a basis for the approximation spaces $V_h$ and $Q_h$, Problem~\ref{problem:semidiscrete} 
can be transformed into a regular linear system of ordinary differential equations
\begin{alignat}{2}
M \underline{\dot u}(t) - B^\top \underline p(t) &= 0,  \qquad & \underline u(0) = \underline{u}_0, \label{eq:lin1}\\
D \underline{\dot p}(t) + B \underline u(t) &= 0, & \underline p(0) = \underline{p}_0, \label{eq:lin2}
\end{alignat}
and the existence of a unique solution $(u_h,p_h)$ follows from the Picard-Lindelöf theorem. 
As shown in \cite{WheelerYotov06}, an appropriate choice for the basis functions for the spaces $V_h$ and $Q_h$ 
leads to (block)-diagonal mass matrices $M$ and $D$. The replacement of the original scalar product $(u_h,v_h)_\Omega$ by the modified one $(u_h,v_h)_{h,\Omega}$ therefore yields a \emph{mass-lumping} strategy which allows us to integrate the system \eqref{eq:lin1}--\eqref{eq:lin2} efficiently in time. 
\end{remark}
The remainder of this section is devoted to the analysis of the semi-discrete method.

\subsection{Auxiliary results}

A common strategy in the analysis of Galerkin approximations of time dependent problems, see e.g. \cite{CowsarDupontWheeler96,DouglasDupontWheeler78,Dupont73}, is to split the error via
\begin{align} \label{eq:split}
\|u - u_h\| \le \|u - u_h^*\| +  \|u_h^* - u_h\| 
\quad \text{and} \quad 
\|p - p_h\| \le \|p - p_h^*\| +  \|p_h^* - p_h\|
\end{align}
into an \emph{approximation error} and a \emph{discrete error} component.
Our analysis relies on a particular choice for the functions $u_h^*$ and $p_h^*$, which is based on the following construction.
\begin{problem}[Inexact elliptic projection] \label{problem:inexactelliptic} $ $ \\  
For any $w \in H(\div;\Omega)$ and $r \in L^2(\Omega)$, 
find $(w_h,r_h)\in V_h \times Q_h$ such that 
\begin{alignat}{2}
(w_h,v_h)_{h,\Omega}-(r_h,\div\, v_h)_\Omega &= (w,v_h)_\Omega - (r,\div\,v_h)_\Omega &&\qquad \forall v_h\in V_h, \label{eq:syswy1} \\
            (\div\,w_h,q_h)_\Omega &= (\div\,w,q_h)_\Omega        &&\qquad \forall q_h\in Q_h. \label{eq:syswy2} 
\end{alignat}
\end{problem}
The existence of a unique solution $(w_h,r_h)$ follows from \eqref{eq:norm} and the choice of the spaces $V_h$ and $Q_h$, see e.g. \cite[Proposition 7.1.1]{BoffiBrezziFortin13} for details. 
For our subsequent analysis, we further require some geometric regularity conditions, i.e., 
\begin{itemize}
\item[(A1)] the mesh $\Th$ is uniformly shape regular, i.e. there exists a $\gamma>0$ such that the diameter $h_K$ of an element $K$ and the radius $\rho_K$ of the largest ball inscribed in $K$ satisfy $\gamma h_K\le\rho_K\le h_K$ uniformly for all elements $K\in\Th$.\label{item:assumption1}
\end{itemize}
This assumption is standard in finite element analysis \cite{ErnGuermond}. 
Since some of our results are based on duality arguments, we introduce as a second standing assumption that
\begin{itemize}
\item[(A2)] $\Omega$ is convex.
\end{itemize}
Note that this condition will be required only for part of our results. 
With these two assumptions at hand, one can prove the following approximation properties.
\begin{lemma}[Estimates for the inexact elliptic projection] \label{lemma:inexact} $ $\\
Let (A1) hold and let $(w_h,r_h)$ denote the solution of Problem~\ref{problem:inexactelliptic}. Then
\begin{align*}
\|w - w_h\|_{L^2(\Omega)} + \|\pi_h^0 r\ - r_h\|_{L^2(\Omega)}&\le C h \|w\|_{H^1(\Omega)}, \\
\|r-r_h\|_{L^2(\Omega)} &\le C h \big(\|w\|_{H^1(\Omega)} + \|r\|_{H^1(\Omega)}\big),
\end{align*}
whenever $w$ and $r$ are sufficiently smooth. 
If also (A2) holds, then additionally
\begin{align*}
\|\pi_h^0 r\ - r_h\|_{L^2(\Omega)} \le C h^2 \big( \|w\|_{H^1(\Omega)} + \|\div\,w\|_{H^1(\Omega)} \big).
\end{align*}
Here $\pi_h^0$ denotes the standard $L^2$-projection onto $Q_h$. 
The constants $C$ in the estimates only depend on the domain $\Omega$ and the shape regularity constant $\gamma$.
\end{lemma}
\begin{proof}
The assertions follow directly from the results of Wheeler and Yotov in \cite{WheelerYotov06}.
\end{proof}

\subsection{Auxiliary functions and approximation error estimates}

We can now construct the auxiliary functions $u_h^*$ and $p_h^*$ 
needed for our analysis as follows.
\begin{problem}[Auxiliary functions] \label{problem:auxfunctions} $ $\\
Let $(u,p)$ denote the unique solution of problem \eqref{eq:sysm1}--\eqref{eq:sysm4}.\\
Then find $u_h^* \in C^1(0,T;V_h)$, $p_h^* \in C(0,T;Q_h)$, and $r_h^*(0) \in Q_h$ satisfying
\begin{alignat*}{2}
(u_h^*(0),v_h)_{h,\Omega} - (r_h^*(0), \div\,v_h)_\Omega &= (u(0),v_h)_\Omega \qquad &&\forall v_h \in V_h, \\
(\div \, u_h^*(0),q_h)_\Omega &= (\div \, u(0), q_h)_\Omega \qquad &&\forall q_h \in Q_h,
\end{alignat*}
and such that for all $0\le t\le T$ there holds
\begin{alignat*}{2}
(\dt u_h^*(t),v_h)_{h,\Omega} - (p_h^*(t), \div\,v_h)_\Omega &= 
0  &&\qquad\forall v_h \in V_h, \\
(\div \, \dt u_h^*(t),q_h)_\Omega &= (\div \,\dt u(t), q_h)_\Omega &&\qquad\forall q_h \in Q_h.
\end{alignat*}
\end{problem}
\noindent
Note that due to \eqref{eq:var1}, the third equation in Problem~\ref{problem:auxfunctions} could also be written as
\begin{align*}
(\dt u_h^*(t),v_h)_{h,\Omega} - (p_h^*(t), \div\,v_h)_\Omega=(\dt u(t),v_h)_\Omega - (p(t), \div \, v_h)_\Omega
\qquad \forall v_h \in V_h.
\end{align*}
From the estimates for the inexact elliptic projection stated above 
and some elementary computations, one can directly deduce the following properties. 
\begin{lemma}[Approximation error estimates] \label{lemma:approximations} $ $\\
Let (A1) hold and let $(u_h^*,p_h^*)$ be the solution of Problem~\ref{problem:auxfunctions}. 
Then
\begin{align*}
(\div\,u(t)- \div \,u^*_h(t),q_h)_\Omega &= 0 \qquad\forall q_h\in Q_h.
\end{align*}
Whenever $(u,p)$ is sufficiently smooth, one further has
\begin{align*}
\|u - u_h^*\|_{W^{k,p}(0,T;L^2(\Omega))} + \|p - p_h^*\|_{W^{k,p}(0,T;L^2(\Omega))} &\le C h \|u\|_{W^{k+1,p}(0,T;H^1(\Omega))}.
\end{align*}
If also (A2) holds, then
\begin{align*}
\|\pi_h^0 p - p_h^*\|_{W^{k,p}(0,T;L^2(\Omega))} &\le C h^2 
\big( \|u\|_{W^{k+1,p}(0,T;H^1(\Omega))}+ \|\div \, u\|_{W^{k+1,p}(0,T;H^1(\Omega))} \big),
\end{align*}
for all $k \ge 0$ and $1 \le p \le \infty$. The constant $C$ may depend on $k$ and $p$, but is independent of $u$, $p$ and $h$. 
\end{lemma}
\begin{proof}
The estimates for $k=0$ follow from the specific construction, using Lemma~\ref{lemma:inexact}, and integration over time. The assertions for $k \ge 1$ are obtained by formal differentiation.
\end{proof}
The results of Lemma~\ref{lemma:approximations} also hold for other choices of $u_h^*(0)$ as long as 
\begin{align*}
(\div\,u(0)- \div \,u^*_h(0),q_h)_\Omega &= 0 \qquad\forall q_h\in Q_h
\end{align*}
is satisfied. An example is $u_h^*(0) =  \rho_h u(0)$ where $\rho_h$ denotes the standard interpolation operator for the $\BDM_1$ space \cite{BoffiEtAl08}. 
The special choice used above will however be crucial for the analysis of our post-processing strategy investigated in Sections~\ref{sec:semi-post} and \ref{sec:time-post}.

\subsection{Discrete error}

As a second step in our analysis of the semi-discrete problem, we now derive estimates for the discrete error components in the splitting \eqref{eq:split}. In the sequel, we always assume that $(u,p)$ is a sufficiently smooth solution of \eqref{eq:sysm1}--\eqref{eq:sysm4} and that $(u_h,p_h)$ is the solution of Problem~\ref{problem:semidiscrete} with initial values given by
\begin{align} \label{eq:initial}
u_{h,0}=u_h^*(0) \qquad \text{and} \qquad p_{h,0}=\pi_h^0p(0). 
\end{align}

Based on the special construction of the auxiliary functions $(u_h^*,p_h^*)$, 
we can establish the following estimates for the discrete error contributions. 
\begin{lemma}[Estimate for the discrete error] \label{lemma:est} $ $\\
Let (A1) hold. Then 
\begin{align*}
\|u_h - u_h^*\|_{L^\infty(0,T;L^2(\Omega))} + \|p_h - p_h^*\|_{L^\infty(0,T;L^2(\Omega))}
\leq C \|\pi_h^0 p - p_h^*\|_{W^{1,1}(0,T;L^2(\Omega))}.
\end{align*}
The constant $C$ only depends on the domain and the shape regularity of the mesh.
\end{lemma}
\begin{proof}
Using the variational characterization of $(u,p)$ and $(u_h,p_h)$, 
as well as the definition of $(u_h^*,p_h^*)$, one can see that
for all $v_h \in V_h$ and $q_h \in Q_h$, and all $0 \le t \le T$ there holds
\begin{alignat*}{2}
(\dt u_h(t) - \dt u_h^*(t),v_h)_{h,\Omega}-(p_h(t)-p_h^*(t),\div\, v_h)_\Omega &= 
0, \\
(\dt p_h(t) - \dt p_h^*(t),q_h)_\Omega+(\div\, (u_h(t) - u_h^*(t)),q_h)_\Omega &= (\pi_h^0 \dt p(t) - \dt p_h^*(t), q_h)_\Omega.
\end{alignat*}
Let us emphasize that the right hand side in the first equation is zero, i.e., the error introduced by mass-lumping does not appear in the discrete error equation. 
By choosing $v_h = u_h(t) - u_h^*(t)$ and $q_h=p_h(t) - p_h^*(t)$, 
and adding the two equations, we obtain
\begin{align*}
\frac{1}{2} \frac{d}{dt} \big(\|u_h(t) - u_h^*(t)\|^2_{h,\Omega} &+ \|p_h(t) - p_h^*(t)\|^2_{L^2(\Omega)} \big)\\
& = (\pi_h^0\dt p(t) - \dt p_h^*(t), p_h(t) - p_h^*(t))_\Omega.
\end{align*}
The right hand side can be estimated by the Cauchy-Schwarz inequality. 
Integration with respect to time and using the choice \eqref{eq:initial} for the initial values further leads to
\begin{align*}
& \|u_h(t) - u_h^*(t)\|_{h,\Omega}^2 +\|p_h(t) - p_h^*(t)\|_{L^2(\Omega)}^2\\
& \le \|\pi_h^0p(0) - p_h^*(0)\|_{L^2(\Omega)}^2+ \|p_h - p_h^*\|_{L^\infty(0,t;L^2(\Omega))} \int_0^t \|\pi_h^0 \dt p(s) - \dt p_h^*(s) \|_{L^2(\Omega)} ds. 
\end{align*}
The assertion now follows by taking the maximum over all $t \in [0,T]$ on both sides, 
using the norm equivalence \eqref{eq:norm}, estimating the right hand side by Young's inequality, 
and finally taking the square root in the resulting estimate.
\end{proof}

\subsection{Error estimates for the semi-discretization}

A combination of the auxiliary results of the previous sections already yields the first main result of our paper.

\begin{theorem}[Error estimate for the semi-discretization]$ $\label{theorem:discreteerror}\\
Let (A1) hold and $(u_h,p_h)$ be the solution of Problem~\ref{problem:semidiscrete} with initial values \eqref{eq:initial}. Then 
\begin{align*}
\|u - u_h\|_{L^{\infty}(0,T;L^2(\Omega))} + \|p - p_h\|_{L^{\infty}(0,T;L^2(\Omega))}
\le h C_0(u)
\end{align*}
with $C_0(u)=C\|u\|_{W^{1,1}(0,T;H^1(\Omega))}$. If also (A2) holds, then 
\begin{align*}
\|u_h - u_h^*\|_{L^\infty(0,T;L^2(\Omega))} + \|p_h - p_h^*\|_{L^\infty(0,T;L^2(\Omega))}&
\le h^2  C_1(u)
\end{align*}
with $C_1(u)= C(\|u\|_{W^{2,1}(0,T;H^1(\Omega))} + \|\div\, u\|_{W^{2,1}(0,T;H^1(\Omega))})$, and additionally
\begin{align*}
\|\pi_h^0 p - p_h\|_{L^\infty(0,T;L^2(\Omega))} & \le h^2  C_1(u).
\end{align*}
\end{theorem}
\begin{proof}
The first estimate follows by a standard error splitting, similar to that in \cite{EggerRadu16}, and noting that the quadrature error can be estimated by  \cite{WheelerYotov06}
\begin{align*}
|(u_h,v_h)_{h,\Omega}-(u_h,v_h)_\Omega|\leq C\sum_{K\in\T_h}h_K\|u_h\|_{H^1(K)}\|v_h\|_{L^2(K)}.
\end{align*}
The second estimate can be deduced from Lemma~\ref{lemma:approximations} and Lemma~\ref{lemma:est}, and the splitting of the error according to \eqref{eq:split}. 
By the triangle inequality, we further obtain
\begin{align*}
\|\pi_h^0 p(t) - p_h(t)\|_{L^2(\Omega)}\leq \|\pi_h^0 p(t) - p_h^*(t)\|_{L^2(\Omega)}+\|p_h^*(t) - p_h(t)\|_{L^2(\Omega)}.
\end{align*}
Another application of Lemma~\ref{lemma:approximations} then yields the third estimate of the lemma.
\end{proof}

Let us emphasize that the first estimate in Theorem~\ref{theorem:discreteerror} cannot be improved in general, i.e., the semi-discretization with mass-lumping stated as Problem~\ref{problem:semidiscrete} is only first order accurate; compare with the numerical tests in Section~\ref{sec:num}. 
The second and third estimate of the Theorem, however, indicate that the discrete solution $(u_h,p_h)$ still carries global second order information. This will be the basis for our following considerations.

\section{Post-processing for the semi-discretization} \label{sec:semi-post}

Based on the second order estimates for the discrete error components in Theorem~\ref{theorem:discreteerror}, we are now able to construct piecewise linear improved approximations $\widetilde u_h(t)$ and $\widetilde p_h(t)$ which define true second order approximations for the exact solution. 

\subsection{Post-processing for the pressure}

For post-processing of the pressure, we can simply employ the strategy proposed in \cite{EggerRadu16} for the \emph{conforming} Galerkin approximation of the wave equation. Let us recall that this method is an extension of the post-processing scheme of Stenberg \cite{Stenberg91} for the Poisson problem to the wave equation. 
\begin{problem}[Post-processing for the pressure] $ $\\
For every $0 \le t \le T$, find $\widetilde p_h(t)\in \P_{1}(\Th)$ such that for all $K \in \Th$ there holds
\begin{alignat*}{2}
(\nabla \widetilde p_h(t), \nabla \widetilde q_h)_K 
  &= -(\dt u_h(t), \nabla \widetilde q_h)_K \qquad &&\forall \widetilde  q_h \in \P_1(K), \\
(\widetilde p_h(t), q_h^0)_K
  &= (p_h(t), q_h^0)_K \qquad &&\forall q_h^0 \in \P_0(K).
\end{alignat*}
\end{problem}

Note that the improved pressure $\widetilde p_h(t)$ is piecewise polynomial but discontinuous in general. 
It can thus be computed separately on every element $K \in \Th$ by solving a small linear system of equations. 
For the subsequent analysis, we additionally require that 
\begin{itemize}
  \item[(A3)] the mesh $\Th$ is quasi-uniform, i.e., $\gamma h \le h_K \le h$ for all $K \in \Th$ for some $\gamma>0$. 
\end{itemize}
This assumption is quite natural when considering explicit time integration later on. 
It additionally allows us to make use of inverse inequalities, e.g., 
\begin{align} \label{eq:inverse}
\|\div  \, v_h\|_{L^2(\Omega)} \le c h^{-1} \|v_h\|_{L^2(\Omega)} \qquad \forall v_h \in V_h.
\end{align}
As a preparatory step, we next derive an estimate for the error in the time derivative. 
\begin{lemma}[Error estimate for the time-derivative]\label{lemma:dtu} $ $\\
Let (A1)--(A3) hold. Then
\begin{align*}
\|\dt u-\dt u_h\|_{L^\infty(0,T;L^2(\Omega))} \le h C_1(u)
\end{align*}
with constant $C_1(u)$ of the same form as defined in Theorem~\ref{theorem:discreteerror}.
\end{lemma}
\begin{proof} 
We again start with splitting the error by
\begin{align*}
\|\dt u(t) - \dt u_h(t)\|_{L^2(\Omega)}
&\le \|\dt u(t) - \dt u_h^*(t)\|_{L^2(\Omega)} + \|\dt u_h^*(t) - \dt u_h(t)\|_{L^2(\Omega)} \\
&= (i)+(ii).  
\end{align*}
The first term is covered by Lemma~\ref{lemma:approximations}, which yields
$(i)\leq C h \|\dt u(t)\|_{H^1(\Omega)}$.
With the aid of the norm equivalence estimate \eqref{eq:norm}, the second term can be bounded by
\begin{align*}
|(ii)|^2
&\le \|\dt u_h^*(t) - \dt u_h(t)\|_{h,\Omega}^2
= (p_h^*(t) - p_h(t),\div\, \dt u_h^*(t) - \div\,\dt u_h(t))_\Omega \\
&\le c h^{-1} \|p_h^*(t) - p_h(t)\|_{L^2(\Omega)} \|\dt u_h^*(t) - \dt u_h(t)\|_{L^2(\Omega)}.
\end{align*}
In the last step, we used the inverse inequality \eqref{eq:inverse} to eliminate the divergence operator.
Together with the estimates of Theorem~\ref{theorem:discreteerror}, we thus obtain
\begin{align*}
(ii) \le  c h^{-1} \|p_h^*(t) - p_h(t)\|_{L^2(\Omega)}
\le h  C_1(u),
\end{align*}
which holds for all $t \in [0,T]$. A combination of the two estimates proves the assertion. 
\end{proof}
Let us note that the estimate for the time derivative could be shown also without assumptions (A2)--(A3); 
see \cite[Lemma 3.8]{EggerRadu16}. Inverse inequalities will however be used also later on in our analysis
and we therefore used them already in the previous proof.
With the help of Lemma~\ref{lemma:dtu}, we can now prove the following estimate.
\begin{theorem}[Error estimate for the improved pressure]\label{theorem:semidiscrete-pp} $ $\\
Let (A1)--(A3) hold. Then
\begin{align*}
\|p-\widetilde p_h\|_{L^\infty(0,T;L^2(\Omega))} \le h^{2} C_1(u),
\end{align*}
with the same constant $C_1(u)$ as defined in Theorem~\ref{theorem:discreteerror}.
\end{theorem}
\begin{proof}
Making use of Lemma~\ref{lemma:dtu} to estimate the error in the time derivative of the velocity,
the proof of the corresponding result in \cite{EggerRadu16} can be applied verbatim. 
\end{proof}

\subsection{Post-processing for the velocity}

We now propose a post-processing strategy for the velocity. 
Due to the presence of the mass-lumping, its analysis is more involved than that for the pressure.
The improved approximation $\widetilde u_h$ is here defined as follows. 
\begin{problem}[Post-processing strategy for the velocity] $ $\\
For every $0 \le t \le T$, find $\widetilde u_h(t)\in V_h$ and $\widetilde r_h(t)\in Q_h$ such that
\begin{alignat*}{2}
(\widetilde u_h(t), v_h)_\Omega - (\widetilde r_h(t), \div\,v_h)_\Omega &= (u_h(t), v_h)_{h,\Omega} \qquad &&\forall v_h \in V_h, \\
(\div\, \widetilde u_h(t), q_h)_\Omega &= (\div\,u_h(t), q_h)_\Omega                     \qquad &&\forall q_h \in Q_h. 
\end{alignat*}
\end{problem}
The post-processing procedure for $u_h$ is local in time but requires the solution of a global system in space. 
By appropriate preconditioning, the improved approximation can, in principle, still be computed in optimal complexity. 
The remainder of this section is devoted to the proof of the following result. 
\begin{theorem}[Error estimate for the improved velocity] \label{theorem:postu} $ $\\
Let (A1)--(A3) hold. Then
\begin{align*}
\|u - \widetilde u_h\|_{L^\infty(0,T;L^2(\Omega))}
\le h^2 \big( C_1(u) + C_2(u)\big) . 
\end{align*}
with constant $C_1(u)$ as in Theorem~\ref{theorem:discreteerror} and $C_2(u)=C \|u\|_{W^{1,1}(0,T;H^{2}(\Omega))}$.
\end{theorem}
To accomplish the proof of this theorem, we will require some preparatory results. 
In particular, we will make use of another auxiliary approximation.
\begin{problem}[Second choice for auxiliary functions]$ $\label{problem:auxfunctions2}\\
Find $\widetilde u_h^{\,*} \in C^1(0,T;V_h)$, $\widetilde p_h^{\,*} \in C(0,T;Q_h)$, and $\widetilde r_h^{\,*}(0) \in Q_h$ satisfying
\begin{alignat*}{2}
(\widetilde u_h^{\,*}(0),v_h)_\Omega - (\widetilde r_h^{\,*}(0), \div\,v_h)_\Omega &= (u(0),v_h)_\Omega, \qquad &&\forall v_h \in V_h, \\
(\div \, \widetilde u_h^{\,*}(0),q_h)_\Omega &= (\div \, u(0), q_h)_\Omega, \qquad &&\forall q_h \in Q_h,
\end{alignat*}
and such that for all $0\le t\le T$ and for all $v_h \in V_h$ and $q_h \in Q_h$ there holds
\begin{alignat*}{2}
(\dt\widetilde u_h^{\,*}(t),v_h)_\Omega - (\widetilde p_h^{\,*}(t), \div\,v_h)_\Omega &= 0,\\
(\div \, \dt \widetilde u_h^{\,*}(t),q_h)_\Omega &= (\div \,\dt u(t), q_h)_\Omega.
\end{alignat*}
\end{problem}
The functions $(\widetilde u_h^{\,*},\widetilde p_h^{\,*})$ are defined similarly as  
$(u_h^{\,*},p_h^{\,*})$, but using the standard scalar product instead of the one with mass-lumping for the first term in the first line.
By well-known results for mixed finite element methods, 
one obtains the following assertions. 
\begin{lemma}[Approximation error estimates] \label{lemma:approximations2} $ $\\
Let (A1) hold and let $(\widetilde u_h^{\,*},\widetilde p_h^{\,*})$ be defined by Problem~\ref{problem:auxfunctions2}. Then
\begin{align*}
(\div\,u(t)-\div\,\widetilde u_h^{\,*}(t),q_h)_\Omega &= 0 \qquad\forall q_h\in Q_h.
\end{align*}
If $u$ and $p$ are sufficiently smooth, then 
\begin{align*}
\|u - \widetilde u_h^{\,*}\|_{L^\infty(0,T;L^2(\Omega))} &\le h^2 C_2(u)
\end{align*}
with constant $C_2(u)$ of the same form as defined in Theorem~\ref{theorem:postu}.
\end{lemma}
\begin{proof}
From standard arguments, see e.g. \cite[Proposition~7.1.2]{BoffiBrezziFortin13}, we know that 
\begin{align*}
\|u(0) - \widetilde u_h^{\,*}(0)\|_{L^2(\Omega)} 
& \le h^2 \|u(0)\|_{H^2(\Omega)}, \\
\|\dt u(t) -  \dt \widetilde u_h^{\,*}(t)\|_{L^2(\Omega)} & \le h^2 \|\dt u(t)\|_{H^2(\Omega)}.
\end{align*}
The results can then be obtained by integration and elementary computations.
\end{proof}

\subsection{Proof of Theorem~\ref{theorem:postu}}

We are now in the position to present the proof of the error estimate for the improved velocity. 
We start by decomposing the error via
\begin{align*}
\|u(t) - \widetilde u_h(t)\|_{L^2(\Omega)} 
\le  \|u(t)-\widetilde u_h^{\,*}(t)\|_{L^2(\Omega)} + \|\widetilde u_h^{\,*}(t)-\widetilde u_h(t)\|_{L^2(\Omega)} 
= (i) + (ii).
\end{align*}
Lemma~\ref{lemma:approximations2} allows us to bound the first term by
\begin{align*}
(i) = \|u(t)-\widetilde u_h^{\,*}(t)\|_{L^2(\Omega)} \le  h^2 C_2(u).
\end{align*}
In order to deal with the second term $(ii)$ in the above estimate, 
first observe that  
\begin{align*}
(\widetilde u_h^{\,*}(t) - \widetilde u_h(t),v_h)_\Omega 
&= (\widetilde u_h^{\,*}(0) - \widetilde u_h(0),v_h)_\Omega
 + \int_0^t (\dt \widetilde u_h^{\,*}(s) - \dt \widetilde u_h(s), v_h)_\Omega \,ds 
\\ &
= (iii) + (iv).
\end{align*}
Using the definitions of $u_h^*(0)$ and $\widetilde u_h^{\;*}(0)$ and the choice \eqref{eq:initial} for the initial conditions, the first term on the right hand side can be expanded in the following way 
\begin{align*}
(iii)
&= (u(0),v_h)_\Omega -(u_h(0),v_h)_{h,\Omega} +(\widetilde r_h^{\,*}(0)-\widetilde r_h(0),\div \, v_h)_\Omega\\
&=(u(0),v_h)_\Omega -(u_h^*(0),v_h)_{h,\Omega} +(\widetilde r_h^{\,*}(0)-\widetilde r_h(0),\div \, v_h)_\Omega\\
&=(\widetilde r_h^{\,*}(0)-\widetilde r_h(0)-r_h^*(0),\div \, v_h)_\Omega.
\end{align*}
The term (iii) thus obviously vanishes, if $\div\,v_h=0$. 
We next turn to the term $(iv)$ and recall that by definition of $\widetilde u_h$ and $\widetilde u_h^{\,*}$, we have
\begin{align*}
(\dt \widetilde u_h^{\,*}(s)-\dt \widetilde u_h(s),v_h)_\Omega &= (\widetilde p_h^{\,*}(s), \div \, v_h)_\Omega
 -(\dt u_h(s),v_h)_{h,\Omega} - (\dt\widetilde r_h(s),\div \,v_h)_\Omega \\
&=(\widetilde p_h^{\,*}(s) - p_h(s) - \dt \widetilde r_h(s), \div \, v_h)_\Omega
\end{align*}
This allows us to express the term $(iv)$ as
\begin{align*}
	(iv)=\int_0^t (\widetilde p_h^{\,*}(s) - p_h(s) - \dt \widetilde r_h(s), \div \, v_h)_\Omega\,ds .
\end{align*}
Again, this part vanishes if $\div \, v_h = 0$. 
In summary, we thus have shown that 
\begin{align} \label{eq:est2a}
(\widetilde u_h^{\,*}(t) - \widetilde u_h(t),v_h)_\Omega = 0 \qquad \text{ if } \div\, v_h = 0.
\end{align}
We can now expand the second term in the first estimate of the proof by 
\begin{align*}
|(ii)|^2 
&= \|\widetilde u_h^{\,*}(t) - \widetilde u_h(t)\|^2_{L^2(\Omega)} \\
&= (\widetilde u_h^{\,*}(t) - \widetilde u_h(t), \widetilde u_h^{\,*}(t) - \widetilde u_h(t) + u_h(t) - u_h^*(t))_\Omega \\
&\qquad \qquad  + (\widetilde u_h^{\,*}(t) - \widetilde u_h(t), u_h^*(t) - u_h(t))_\Omega = (v) + (vi).
\end{align*}
From the particular construction of the auxiliary functions $u_h^*$ and $\widetilde u_h^{\;*}$, and of the improved approximation $\widetilde u_h$, one can deduce that 
\begin{align*}
\div \,\widetilde u_h^{\,*}(t) = \pi_h^0 \div\, u(t) = \div\, u_h^*(t)\quad\text{ and }\quad\div\, \widetilde u_h(t) = \div\, u_h(t),
\end{align*}
for all $t>0$. Therefore, the test function in term $(v)$ has zero divergence and according to equations \eqref{eq:est2a}, we 
obtain $(v)=0$. 
The term $(vi)$ can be further estimated by the Cauchy-Schwarz inequality and Theorem~\ref{theorem:discreteerror}, which finally yields
\begin{align*}
(ii) \le \|u_h(t) - u_h^*(t)\|_{L^2(\Omega)} \le h^2 C_1(u).
\end{align*}
This concludes the proof of our second main result concerning the semi-discretization.
\qed

\section{Time discretization} \label{sec:time}

Our main motivation for the investigation of mass-lumping for the mixed finite element method was to enable an efficient numerical integration in time by explicit Runge-Kutta or multistep methods.
We here consider the leapfrog scheme, which is explicit and formally second order accurate and which enjoys many further interesting properties \cite{HairerLubichWanner03,Joly03}. 

\subsection{The fully discrete scheme}

Choose $N>1$, set $\tau = T/N$, and define time steps 
$t^n=n\tau$ and $t^{n-1/2} = (n-1/2) \tau$ for all $0\le n\le N$. 
We use the symbols
\begin{align*}
\dtau q_h^{n+1/2} = \frac{q_h^{n+1} - q_h^{n}}{\tau}
\qquad \text{and} \qquad
\dtau v_h^{n} = \frac{v_h^{n+1/2}-v_h^{n-1/2}}{\tau}
\end{align*}
to denote the backward difference quotients for given sequences 
$\{q^n_h\}_{n\ge 0}$ and $\{v_h^{n-1/2}\}_{n \ge 0}$.
For the time-discretization of the semi-discrete problem \eqref{eq:sysvd1}--\eqref{eq:sysvd3},
we then consider 
\begin{problem}[Fully discrete problem]$ $ \label{problem:fulldiscrete}\\
Set $p_h^{0} = \pi_h^0 p(0)$ and define $u_h^{-1/2} \in V_h$ as solution of 
\begin{align}
(u_h^{-1/2},v_h)_{h,\Omega} = (u_h^*(0),v_h)_{h,\Omega} - \frac{\tau}{2}  (p_h^0,\div \, v_h)_\Omega \qquad \forall v_h \in V_h. \label{eq:sysvdf1}
\end{align}
Then for $0 \le n \le N-1$ find $(u_h^{n+1/2},p_h^{n+1}) \in V_h\times Q_h$, such that 
\begin{alignat}{2}
(\dtau u_h^{n},v_h)_{h,\Omega} - (p_h^{n}, \div \, v_h)_\Omega 
     &= 0 \qquad &&\forall v_h \in V_h, \label{eq:sysvdf2} \\
(\dtau p_h^{n+1/2},q_h)_\Omega + (\div \, u_h^{n+1/2},q_h)_\Omega 
     &= 0 \qquad &&\forall q_h \in Q_h. \label{eq:sysvdf3}
\end{alignat}
\end{problem}
Let us note that as a consequence of the norm equivalence estimate \eqref{eq:norm}, 
the fully discrete scheme can be seen to be well-defined.
The remainder of this section will be devoted to the error analysis of the proposed method.

\subsection*{Convention}
We denote by $(u,p)$ a sufficiently smooth solution of problem \eqref{eq:sysm1}--\eqref{eq:sysm4} and by $(u_h^{n-1/2},p_h^n)_{n \ge 0}$ the unique solution of Problem~\ref{problem:fulldiscrete}. 
In addition, we assume that the conditions (A1)--(A3) are valid and that 
the time step $\tau$ satisfies the CFL condition
\begin{itemize}
\item[(A4)] $0<\tau \le 1/2$ and $\tau \le h/c$ where $c$ is the constant of the inverse inequality \eqref{eq:inverse}. 
\end{itemize}
As a direct consequence of \eqref{eq:inverse} and assumption (A4), we obtain that
\begin{align} \label{eq:inverse2}
\|\div\,v_h\|_{L^2(\Omega)} \le \frac{1}{\tau}\|v_h\|_{h,\Omega} \qquad\forall\,v_h \in V_h,
\end{align}
which will be used in the proof of discrete energy estimates in the next section.

\subsection{Discrete error}

Let us start by introducing some additional notations.
For any sequence $\{v_h^{n-1/2}\}_{n \ge 0}$ and any (piecewise) continuous function 
$z$, we denote by
\begin{align*}
\widehat v_h^{\,n} = \frac{v_h^{n+1/2}+v_h^{n-1/2}}{2} 
\qquad \text{and} \qquad 
\widehat z(t) := \frac{1}{\tau} \int_{t-\tau/2}^{t+\tau/2} z(s) ds
\end{align*}
local averages around $t^n$ and $t$, respectively. 
Using the auxiliary functions $(u_h^*,p_h^*)$ introduced in Problem~\ref{problem:auxfunctions}, we can split the error into an approximation error and a discrete error component.
Similar to the semi-discrete level, the discrete error component can be characterized as the solution of a particular fully discrete system. 
\begin{lemma}[Discrete error equation]\label{lemma:fulldiscrete} $ $\\
For all $1 \le n \le N$, define 
\begin{align*}
a_h^{n-1/2} = u_h^{n-1/2} - u_h^*(t^{n-1/2})
\qquad \text{and} \qquad 
b_h^{n} = p_h^{n} - \widehat p_h^{\,*}(t^{n}).
\end{align*}
Furthermore, let $a_h^{-1/2} \in V_h$ and $b_h^{0}\in Q_h$ be the unique solution of  
\begin{alignat*}{2}
(a_h^{-1/2},v_h)_{h,\Omega} &= (a_h^{1/2},v_h)_{h,\Omega} - {\tau} (b_h^0, \div \, v_h)_\Omega \qquad &&\forall v_h \in V_h,\\
(b_h^{0},q_h)_\Omega &= (b_h^{1},q_h)_{h,\Omega} + {\tau} ( \div \,a_h^{1/2}, q_h)_\Omega \qquad &&\forall q_h \in Q_h.
\end{alignat*}
Then the sequence $\{(a_h^{n-1/2},b_h^{n})\}_{0 \le n \le N} \subset V_h\times Q_h$ satisfies
\begin{alignat}{2}
(\dtau a_h^{n},v_h)_{h,\Omega} - (b_h^{n}, \div \, v_h)_\Omega &= 0 \qquad &&\forall v_h \in V_h, \label{eq:id1}\\
(\dtau b_h^{n+1/2},q_h)_\Omega + (\div \, a_h^{n+1/2},q_h)_\Omega &= (g_h^{n+1/2}, q_h)_\Omega \qquad &&\forall q_h \in Q_h, \label{eq:id2}
\end{alignat}
for all $0 \le n \le N-1$ with right hand side $g_h^{1/2}=0$ and 
\begin{align}
g_h^{n+1/2} 
= \pi_h^0 \dt p(t^{n+1/2}) - \dtau \widehat p_h^{\,*}(t^{n+1/2}) \qquad \text{for all } 1 \le n \le N-1. \label{eq:id3}
\end{align}
\end{lemma}
\begin{proof} 
The two identities for $n=0$ follow directly from the construction of $a_h^{-1/2}$ and $b_h^{\,0}$. 
Now let $n \ge 1$. Then by equation \eqref{eq:sysvdf1}, 
using the fundamental theorem of calculus, and the definition of $(u_h^*,p_h^*)$, one can see that
\begin{align*}
(\dtau a_h^n,v_h)_{h,\Omega} &- (b_h^n, \div \, v_h)_\Omega  \\
&= \frac{1}{\tau}(u_h^*(t^{n+1/2}) - u_h^*(t^{n-1/2}),v_h)_{h,\Omega} - 
\frac{1}{\tau} \int_{t^{n-1/2}}^{t^{n+1/2}} (p_h^*(s), \div \, v_h)_\Omega\, ds \\
&= \frac{1}{\tau} \int_{t^{n-1/2}}^{t^{n+1/2}} (\dt u_h^*(s), v_h)_{h,\Omega} - (p_h^*(s), \div \, v_h)_\Omega\, ds = 0.
\end{align*}
This already proves the first identity \eqref{eq:id1} for $n \ge 1$.
From the characterizations of the discrete and continuous solutions and the definition of the auxiliary functions, we get 
\begin{align*}
(\dtau b_h^{n+1/2},q_h)_\Omega &+ (\div \, a_h^{n+1/2}, q_h)_\Omega \\
&=(\dt p(t^{n+1/2})-\dtau \widehat p_h^{\,*}(t^{n+1/2}), q_h)_\Omega + (\div \, u(t^{n+1/2})-\div \, u_h^*(t^{n+1/2}), q_h)_\Omega \\
&=(\dt p(t^{n+1/2})-\dtau \widehat p_h^{\,*}(t^{n+1/2}), q_h)_\Omega.
\end{align*}
Note that the divergence terms here cancel out by the special construction of $u_h^*$; see Lemma~\ref{lemma:approximations}. 
This shows the second identity \eqref{eq:id2} of the Lemma for $n \ge 1$.
\end{proof}
As a next step, we now derive an energy estimate for the fully discrete scheme. 
\begin{lemma}[Discrete energy estimate] $ $\label{lemma:discenergyest}\\ 
Let $\{a_h^{n-1/2}\} \subset V_h$, $\{b_h^{n}\} \subset Q_h$, and  $\{g_h^{n+1/2}\} \subset Q_h$ be given sequences such that 
\begin{alignat}{2}
(\dtau a_h^{n},v_h)_{h,\Omega} - (b_h^{n}, \div \, v_h)_\Omega &= 0 \qquad &&\forall v_h \in V_h, \label{eq:dee1}\\
(\dtau b_h^{n+1/2},q_h)_\Omega + (\div \, a_h^{n+1/2},q_h)_\Omega &= (g_h^{n+1/2}, q_h)_\Omega \qquad &&\forall q_h \in Q_h. \label{eq:dee2}
\end{alignat}
Furthermore, assume that (A3)--(A4) hold. Then
\begin{align}
\|\widehat a_h^{\,n}\|_{L^2(\Omega)}^2 + \|b_h^n\|_{L^2(\Omega)}^2\le 
C_T \big( \|\widehat a_h^{\,0}\|_{L^2(\Omega)}^2 +\|b_h^0\|_{L^2(\Omega)}^2+
\sum_{k=0}^{n-1} \tau \|g_h^{k+1/2}\|_{L^2(\Omega)}^2 \big), \label{eq:discenergyest}
\end{align}
for all $0 \le n \le N-1$ with $\widehat a_h^{\,n} = \frac{1}{2}(a_h^{n+1/2}+a_h^{n-1/2})$ 
and $C_T$ depending on $T=N \tau$.
\end{lemma}
\begin{proof}
We essentially follow the arguments presented in \cite{Joly03}. 
Testing equation \eqref{eq:dee1} for index $(n+1)$ and index $n$ with $v_h=\tau a_h^{n+1/2}$
and summing up the two results leads to
\begin{align*}
0 &= \tau (\dtau a_h^{n+1} + \dtau a_h^n, a_h^{n+1/2})_\Omega - \tau (b_h^{n+1} + b_h^n, \div\,a_h^{n+1/2})_\Omega.
\end{align*}
The first term on the right hand side can be evaluated via
\begin{align*}
\tau  (\dtau a_h^{n+1} &+ \dtau a_h^n, a_h^{n+1/2})_\Omega
= (a_h^{n+3/2} - a_h^{n-1/2},a_h^{n+1/2})_\Omega  \\
&= \|\widehat a_h^{\,n+1}\|_{L^2(\Omega)}^2 - \frac{\tau^2}{4} \|\dtau a_h^{n+1}\|_{L^2(\Omega)}^2 -
 \|\widehat a_h^{\,n}\|_{L^2(\Omega)}^2 + \frac{\tau^2}{4} \|\dtau a_h^n\|_{L^2(\Omega)}^2.
\end{align*}
As a next step, we test the second equation \eqref{eq:dee2} with $q_h=\tau (b_h^{n+1} + b_h^n)$, 
which yields 
\begin{align*} 
\tau (g_h^{n+1/2},b_h^{n+1} + b_h^n)_\Omega
&= \tau (\dtau b_h^{n+1/2},b_h^{n+1} + b_h^n)_\Omega + \tau (\div \, a_h^{n+1/2}, b_h^{n+1} + b_h^n)_\Omega\\
  &= \|b_h^{n+1}\|^2_{L^2(\Omega)} - \|b_h^n\|^2_{L^2(\Omega)} + \tau (\div\,a_h^{n+1/2}, b_h^{n+1} + b_h^n)_\Omega. \notag
\end{align*}
Summing up these intermediate results, we arrive at 
\begin{align}
\|\widehat a_h^{\,n+1}\|_{h,\Omega}^2 &+ \|b_h^{n+1}\|_{L^2(\Omega)}^2 - 
\frac{\tau^2}{4} \|\dtau a_h^{n+1}\|_{h,\Omega} \label{eq:discenergy}\\
&= \|\widehat a_h^{\,n}\|_{h,\Omega}^2 + \|b_h^n\|_{L^2(\Omega)}^2 - \frac{\tau^2}{4} \|\dtau a_h^n\|_{h,\Omega}^2 \notag
 + \tau (g_h^{n+1/2},b_h^{n+1} + b_h^n)_\Omega.
\end{align}
Using equation \eqref{eq:dee1} and the estimate \eqref{eq:inverse2} leads to
\begin{align*}
\|\dtau a_h^n\|_{h,\Omega}^2 &= (b_h^n,\div \, \dtau a_h^n)_\Omega \le \frac{1}{\tau} \|b_h^n\|_{L^2(\Omega)} \|\dtau a_h^n\|_{h,\Omega}. 
\end{align*}
This shows that the energy $E_h^n := \|\widehat a_h^{\,n}\|_{h,\Omega}^2 + \|b_h^n\|_{L^2(\Omega)}^2 - 
\frac{\tau^2}{4} \|\dtau a_h^n\|_{h,\Omega}^2$ is positive with
\begin{align} \label{eq:equivalence}
E_h^n \le \|\widehat a_h^{\,n}\|^2_{h,\Omega} + \|b_h^n\|^2_{L^2(\Omega)} \le 2 E_h^n. 
\end{align}
By the Cauchy-Schwarz and Young inequalities, we can then further estimate
the last term in the identity \eqref{eq:discenergy} by
\begin{align*}
\tau (g_h^{n+1/2},b_h^{n+1} + b_h^n)_\Omega
&\le \tau \|g_h^{n+1/2}\|^2_{L^2(\Omega)} + \frac{\tau}{2} \|b_h^{n+1}\|^2_{L^2(\Omega)}  + \frac{\tau}{2} \|b_h^{n}\|^2_{L^2(\Omega)} \\
&\le \tau \|g_h^{n+1/2}\|^2_{L^2(\Omega)} + \tau E_h^{n+1} + \tau E_h^{n}.
\end{align*}
Inserting this estimate into equation \eqref{eq:discenergy}, we therefore obtain 
\begin{align*}
(1-\tau) E_h^{n+1} \le (1+\tau) E_h^{n} + \tau \|g_h^{n+1/2}\|_{L^2(\Omega)}^2. 
\end{align*}
A recursive application of this inequality finally leads to 
\begin{align*}
E_h^{n} &\le \left(\tfrac{1+\tau}{1-\tau}\right)^n  E_h^0 + \tfrac{\tau}{1-\tau}
\sum_{k=0}^{n-1} \left(\tfrac{1+\tau}{1-\tau}\right)^{n-k-1}  \|g_h^{n+1/2}\|_{L^2(\Omega)}^2 \\
&\le e^{2n\tau} \left(E_h^0 + 2 \sum_{k=0}^{n-1} \tau \|g_h^{n+1/2}\|^2_{L^2(\Omega)} \right).
\end{align*}
Here we used the condition $\tau \le 1/2$ in the last step. 
The assertion of the Lemma now follows by using the norm equivalence estimate \eqref{eq:equivalence}.
\end{proof}

\subsection{Taylor estimates}

As a next step in our analysis, 
we now derive bounds for the initial errors and residuals arising in the right-hand side of the estimate \eqref{eq:discenergyest}. 
Let us start by estimating the errors in the initial conditions.
\begin{lemma} [Taylor estimates, part one] \label{lemma:taylorest1} 
Let (A1)--(A4) hold. Then
\begin{align*}
\|\widehat a_h^{\,0}\|_{L^2(\Omega)} + \|b_h^{\,0}\|_{L^2(\Omega)} 
& \le (h^2+\tau^2) C_1(u),
\end{align*}
with constant $C_1(u)$ of the form as defined in Theorem~\ref{theorem:discreteerror}.
\end{lemma}
\begin{proof}
By definition of $b_h^0$, we have 
\begin{align*}
b_h^0&=b_h^1+\tau\,\div\,a_h^{1/2}
=p_h^1-\widehat p_h^{\,*}(t^1)+\tau\,\div\,\big(u_h^{1/2}-u_h^*(t^{1/2})\big)\\
&=p_h^1-p_h^0 + \tau\,\div\, u_h^{1/2} +p_h^0-\widehat p_h^{\,*}(t^1)-\tau\,\pi_h^0 \div\,u(t^{1/2})\\
&=\pi_h^0 p(0)-\widehat p_h^{\,*}(t^1)+\tau\,\pi_h^0 \dt p(t^{1/2})\\
&=\big(\pi_h^0 p(0)-\pi_h^0 \widehat p(t^1)+\tau\,\pi_h^0 \dt p(t^{1/2})\big) + 
\big(\pi_h^0\widehat p(t^1)-\widehat p_h^{\,*}(t^1)\big)
=(i)+(ii).
\end{align*}
Using standard Taylor estimates, the first term can be estimated by
\begin{align*}
\|(i)\|_{L^2(\Omega)} 
&\leq \tau^2 \|\dtt p\|_{L^\infty(0,T;L^2(\Omega))} 
 = \tau^2 \|\div \; \dt u\|_{L^\infty(0,T;L^2(\Omega))}
 \le \tau^2 C_1(u),
\end{align*}
where we used equation \eqref{eq:sysm2} in the second step and 
 $\|\cdot\|_{L^\infty(0,T;X)} \le C \|\cdot\|_{W^{1,1}(0,T;X)}$ by 
continuity of the embedding.
The second term can be further estimated by Lemma~\ref{lemma:approximations} according to 
$\|(ii)\|\leq \|\pi_h^0 p-p_h^*\|_{L^\infty(0,T;L^2(\Omega))}\le h^2 C_1(u).$
This completes the estimate for $b_h^0$ and we can turn to that for $\widehat a_h^{\,0}$.
Using the definitions in Lemma~\ref{lemma:fulldiscrete}, one can see that
\begin{align} \label{eq:est99}
(\widehat a_h^{\,0},v_h)_{h,\Omega} = \frac{1}{2}(a_h^{1/2}+a_h^{-1/2},v_h)_{h,\Omega} = 
(a_h^{1/2},v_h)_{h,\Omega} - \frac{\tau}{2}(b_h^0,\div\,v_h)_\Omega.
\end{align}
From the construction of $u_h^{-1/2}$ and the first recursion of Problem~\ref{problem:fulldiscrete} for $n=0$, we get
\begin{align*}
(u_h^{1/2},v_h)_{h,\Omega} = (u_h^0,v_h)_{h,\Omega} + \frac{\tau}{2} (p_h^0,\div \, v_h)_\Omega \qquad \forall v_h \in V_h.
\end{align*}
Recalling the definition of $a_h^{1/2}$ given in Lemma~\ref{lemma:fulldiscrete}, we obtain
\begin{align*}
(a_h^{1/2},v_h)_{h,\Omega}
&=(u_h^{1/2}-u_h^*(t^{1/2}),v_h)_{h,\Omega}=
\frac{\tau}{2} (p_h^0,\div \, v_h)_\Omega - \int_0^{t^{1/2}}(\dt u_h^*(s),v_h)_{h,\Omega}\,ds\ \\
&=\frac{\tau}{2} (\pi_h^0p(0)-p_h^*(0),\div \, v_h)_\Omega +\int_0^{t^{1/2}}(p_h^*(0)-p_h^*(s),\div\,v_h)_\Omega\,ds\\
&=\frac{\tau}{2} (\pi_h^0p(0)-p_h^*(0),\div \, v_h)_\Omega -\int_0^{t^{1/2}}\int_0^s\left(\dt p_h^*(r),\div\,v_h\right)_\Omega\,dr\,ds\\
&=(i) + (ii).
\end{align*}
For the second identity, we here used the definition $u_h^0=u_h^*(0)$ of the initial value for the discrete velocity. The first term can be further estimated by
\begin{align*}
|(i)|\le c\tau h^{-1}\|\pi_h^0p(0)-p_h^*(0)\|_{L^2(\Omega)}\|v_h\|_{L^2(\Omega)}\leq h^2 C_1(u)\|v_h\|_{L^2(\Omega)}
\end{align*}
By adding and subtracting $\dt p(r)$ under the integral, using integration-by-parts for one of the terms, and applying the triangle and Cauchy-Schwarz inequalities, we obtain
\begin{align*}
|(ii)|
& \le \tau \|\pi_h^0 \dt p - \dt p_h^*\|_{L^1(0,T;L^2(\Omega))} \|\div\,v_h\|_{L^2(\Omega)}
+ \tau^2 \|\dtt p\|_{L^\infty(0,T;H^1(\Omega))} \|v_h\|_{L^2(\Omega)}.
\end{align*}
With the inverse inequality \eqref{eq:inverse}, Lemma~\ref{lemma:approximations}, 
and equation \eqref{eq:sysm2}, one can further see that
\begin{align*}
|(a_h^{1/2},v_h)_{h,\Omega}|
&\le \big(ch^2 C_1(u) + \tau^2 \|\div \; \dt u\|_{L^\infty(0,T;L^2(\Omega))} \big) \|v_h\|_{L^2(\Omega)}.
\end{align*}
This allows us to estimate the first term in \eqref{eq:est99} by $|(a_h^{1/2},v_h)_{h,\Omega}| \le (h^2+\tau^2) C_1(u) \|v_h\|_{L^2(\Omega)}$.
Using the CFL condition (A4), we can bound the second term in \eqref{eq:est99} by
\begin{align*}
\big|\frac{\tau}{2}(b_h^0,\div\,v_h)_\Omega\big| 
&\le \|b_h^0\|_{L^2(\Omega)} \|v_h\|_{L^2(\Omega)}.
\end{align*}
A combination of the individual estimates for the two terms in \eqref{eq:est99} and the previous estimate for $\|b_h^0\|_{L^2(\Omega)}$ now yields the desired bound for $\widehat a_h^{\,0}$ and completes the proof.
\end{proof}

As a next step, we derive appropriate bounds for the residuals $g_h^{n+1/2}$ in Lemma~\ref{lemma:fulldiscrete}.

\begin{lemma} [Taylor estimates, part two] \label{lemma:taylorest2} 
Let (A1)--(A3) hold. Then 
\begin{align*}
\bigg(\sum_{n=0}^{N-2} \tau \|g_h^{n+1/2}\|_{L^2(\Omega)}^2 \bigg)^{1/2}
& \le (h^2 + \tau^2) C_1(u)
\end{align*}
with constant $C_1(u)$ of the same form as defined in Theorem~\ref{theorem:discreteerror}.
\end{lemma}

\begin{proof}
By elementary computations and the definition of $g_h^{n+1/2}$ in Lemma~\ref{lemma:fulldiscrete}, we have
\begin{align*}
\|g_h^{n+1/2}\|_{L^2(\Omega)} 
&\le \|\dt p(t^{n+1/2}) - \dtau \widehat p(t^{n+1/2})\|_{L^2(\Omega)} 
    + \| \dtau \pi_h^0\widehat p(t^{n+1/2}) - \dtau \widehat p_h^{\,*}(t^{n+1/2})\|_{L^2(\Omega)}\\
&=(i)+(ii).
\end{align*}
Using standard Taylor expansions and Cauchy-Schwarz inequalities, one can verify that
\begin{align*}
|(i)|^2
&\leq \tau^3 \int_{t^{k-1/2}}^{t^{n+3/2}}\|\dttt p(s)\|_{L^2(\Omega)}^2\,ds 
=\tau^3 \int_{t^{n-1/2}}^{t^{n+3/2}}\|\div \; \dtt u(s)\|_{L^2(\Omega)}^2\,ds,
\end{align*}
where we used equation \eqref{eq:sysm2} in the last step. 
Via the fundamental theorem of calculus and Theorem~\ref{theorem:discreteerror}, 
the second term can be further estimated by 
\begin{align*}
|(ii)|^2 
&\leq \frac{1}{\tau}\int_{t^{n-1/2}}^{t^{n+3/2}}\|\pi_h^0 \dt p(s)- \dt p_h^*(s)\|_{L^2(\Omega)}^2\,ds \\ 
&\leq C \frac{h^4}{\tau}\int_{t^{n-1/2}}^{t^{n+3/2}}\|\dtt u(s)\|_{H^1(\Omega)}^2 + \|\div\,\dtt u(s)\|_{H^1(\Omega)}^2\,ds.
\end{align*}
Here we used the estimates of Lemma~\ref{lemma:approximations} in the second step. 
The assertion of the Lemma then follows by summation over all $0 \le n \le N-2$. 
\end{proof}

\subsection{Error estimate for the fully discrete scheme}

A combination of the auxiliary results stated above now allows us to prove the following assertions.
\begin{theorem}[Estimates for the discrete error] \label{theorem:fulldiscreteerror} 
Let (A1)--(A4) hold. Then 
\begin{align*}
\|\widehat u_h^{\,n} - \widehat u_h^{\,*}(t^n)\|_{L^2(\Omega)} 
  + \|p_h^n - \widehat p_h^{\,*}(t^n)\|_{L^2(\Omega)}
\le (h^2+\tau^2) C_1(u)
\end{align*}
for all $0 \le n \le N-1$ with constant $C_1(u)$ as in Lemma~\ref{theorem:discreteerror}. 
In addition, we also have
\begin{align*}
\|p_h^n - p_h^*(t^n)\|_{L^2(\Omega)} + \|p_h^n - \pi_h^0 p(t^n)\|_{L^2(\Omega)} 
&\le (h^2+\tau^2) C_1(u) .
\end{align*}
\end{theorem}
\begin{proof}
A combination of Lemma~\ref{lemma:fulldiscrete}, \ref{lemma:discenergyest}, \ref{lemma:taylorest1},
and \ref{lemma:taylorest2} already yields the first assertion. 
The first term in the second estimate can be bounded by
\begin{align*}
\|p_h^n - p_h^*(t^n)\|_{L^2(\Omega)}
&\le \|p_h^n-\widehat p_h^{\,*}(t^n)\|_{L^2(\Omega)} + \|\widehat p_h^{\,*}(t^n)-p_h^*(t^n)\|_{L^2(\Omega)}=(i)+(ii).
\end{align*}
The first part is readily covered by the first assertion of the theorem. 
Using Taylor expansions and Lemma~\ref{lemma:approximations}, the second part can be further estimated by 
\begin{align*}
(ii) 
&\leq \|\widehat p_h^{\,*}(t^n) - \pi_h^0 \widehat p(t^n)\|_{L^2(\Omega)} + \|\pi_h^0 \widehat p(t^n) - \pi_h^0 p(t^n)\|_{L^2(\Omega)} + \|\pi_h^0 p(t^n) - p_h^{*}(t^n)\|_{L^2(\Omega)} \\
&\le \tau^2 \|\dtt p\|_{L^\infty(0,T;L^2(\Omega))} + h^2 C_1(u)=\tau^2 \|\div\,\dt u\|_{L^\infty(0,T;L^2(\Omega))} + h^2 C_1(u).
\end{align*}
For the second term in the second estimate of the Theorem, we observe that 
\begin{align*}
\|p_h^n - \pi_h^0 p(t^n)\|_{L^2(\Omega)}\le \|p_h^n-p_h^*(t^n)\|_{L^2(\Omega)} + 
\|p_h^*(t^n)-\pi_h^0 p(t^n)\|_{L^2(\Omega)}.
\end{align*}
These terms are covered by Lemma~\ref{lemma:approximations} and Lemma~\ref{lemma:discenergyest}, respectively.
\end{proof}

Similar to the semi-discrete level, we may again also obtain an error estimate
of first order. For completeness, we state the corresponding result explicitly.
\begin{theorem} \label{theorem:fulldiscreteerror2}
Let (A1)--(A4) hold. Then for $0 \le n \le N-1$, we have
\begin{align*}
\|\widehat u_h^{\,n} - u(t^n)\|_{L^2(\Omega)} + \|p_h^n - p(t^n)\|_{L^2(\Omega)} \le (h+\tau) C_1(u).
\end{align*}
\end{theorem}
\begin{proof}
The result follows with similar arguments as used in the previous proof. 
\end{proof}
Let us remark that convergence of first order could be obtained also without assumptions (A2) and (A3) and under less stringent smoothness assumptions. Since we are interested in second order estimates, we do not go into details here.

\section{Post-processing for the full discretization} \label{sec:time-post}

With similar arguments as on the semi-discrete level, we can now
construct improved approximations $\widetilde p_h^{\,n}$ and $\widetilde u_h^{\,n}$ 
which are true second order approximations for the solution. 

\subsection{Post-processing for the pressure}

Using a similar construction as for the semi-discrete problem, 
we now consider the following post-processing procedure.
\begin{problem}[Post-processing for the discrete pressure] $ $\\
For all $0 \le n \le N-1$ find $\widetilde p_h^{\,n} \in \P_{1}(\Th)$ such that
\begin{alignat}{2}
(\grad\widetilde p_h^{\,n},\grad \widetilde q_h)_K&=
-(\dtau\u_h^{n},\grad \widetilde q_h)_K \qquad &&\forall \widetilde q_h\in \P_{1}(K), \label{eq:sysppd1} \\ 
(\widetilde p_h^{\,n},q_h^0)_K&=(p_h^{n},q_h^0)_K &&\forall q_h^0\in \P_0(K) \label{eq:sysppd2}.
\end{alignat}
\end{problem}
As before, the improved approximation $\widetilde p_h^{\,n}$ can be computed by solving small linear systems on every element $K$ independently. 
For the analysis of this post-processing scheme, we again need an estimate for the error in the time derivative of the velocity.
\begin{lemma} [Estimate for the time derivatives] $ $\label{lemma:dtudiscrete}\\
Let (A1)--(A4) hold. 
Then for all $0 \le n \le N-1$, there holds
\begin{align*}
\|\dt u(t^n) - \dtau u_h^n\|_{L^2(\Omega)} \le (h+\tau) C_1(u).
\end{align*}
\end{lemma}
\begin{proof} 
We start with splitting the error in the time derivative by 
\begin{align*}
\|\dt u(t^n) - \dtau u_h^n\|_{L^2(\Omega)}
&\le \|\dt u(t^n) - \dt u_h^*(t^n)\|_{L^2(\Omega)} + \|\dt u_h^*(t^n) - \dtau u_h^n\|_{L^2(\Omega)} \\ &=(i)+(ii).
\end{align*}
The first term is covered by the estimates of Lemma~\ref{lemma:approximations}, which yield
\begin{align*}
	|(i)|\leq C h \|\dt u(t^n)\|_{H^1(\Omega)} \le h C_1(u).
\end{align*}
The norm equivalence \eqref{eq:norm}, the variational characterizations of $u_h^*(t^n)$ and $u_h^n$, and assumption (A3) now allow us to estimate the second term by
\begin{align*}
|(ii)|^2 
&\le \|\dt u_h^*(t^n) - \dtau u_h^n\|_{h,\Omega}^2 
 = (p_h^*(t^n) - p_h^n,\div\,(\dt u_h^*(t^n) - \dtau u_h^n))_\Omega \\
&\le c h^{-1} \|p_h^*(t^n) - p_h^n\|_{L^2(\Omega)}  \|\dt u_h^*(t^n) - \dtau u_h^n\|_{L^2(\Omega)} .
\end{align*}
With the estimates of the Theorem~\ref{theorem:fulldiscreteerror} and assumption (A4), 
we thus obtain
\begin{align*}
|(ii)| \le C h^{-1} \|p_h^*(t^n) - p_h^n\|_{L^2(\Omega)} \le  (h+\tau) C_1(u).
\end{align*}
A combination of the two results yields the assertion.
\end{proof}
Together with the estimates of Theorem~\ref{theorem:fulldiscreteerror}, 
we can now establish the following bound.
\begin{theorem}[Estimate for the improved pressure] \label{theorem:pfullpost} $ $\\
Let $\widetilde p_h^{\,n}$ be defined as above and let (A1)--(A4) hold.
Then, for all $0 \le n \le N-1$, we have
\begin{align*}
\|\widetilde p_h^{\,n} - p(t^n)\|_{L^2(\Omega)} 
\le (h^2 + \tau^2) C_1(u)
\end{align*}
with constant $C_1(u)$ of the same form as stated in Theorem~\ref{theorem:discreteerror}.
\end{theorem}
\begin{proof}
We can proceed in the same manner as on the semi-discrete level; here, Lemma~\ref{lemma:dtudiscrete} gives the estimate for the error of the discrete time derivative of $u$; see also \cite{EggerRadu16}.
\end{proof}

\subsection{Post-processing for the velocity}

Let us now turn to the post-processing of the velocity, for which we again employ a similar construction as on the semi-discrete level.
\begin{problem}[Post-processing for the velocity] $ $\\
For all $0 < n \le N$ find $\widetilde u_h^{\,n} \in V_h$ and $\widetilde r_h^{\,n} \in Q_h$ such that
\begin{alignat*}{2}
(\widetilde u_h^{\,n}, v_h)_\Omega - (\widetilde r_h^{\,n}, \div\,v_h)_\Omega &= (\widehat u_h^{\,n}, v_h)_{h,\Omega}
\qquad && \forall v_h \in V_h, \\
(\div\,\widetilde u_h^{\,n}, q_h)_\Omega &= (\div\,\widehat u_h^{\,n}, q_h)_\Omega \qquad && \forall q_h \in Q_h.
\end{alignat*}
\end{problem}
Recall that $\widehat u_h^{\,n} = \frac{1}{2}(u_h^{n+1/2}+u_h^{n-1/2})$ denotes the average of the discrete solution at two consecutive time points and therefore, $\widetilde u_h^{\,n}$ is actually an approximation for $\widehat u(t^n)$. 
Similar to the semi-discrete level, the post-processing scheme can be computed locally in time, but requires the solution of a global system in space.
We can now establish the following estimate for the improved velocity approximation. 
\begin{theorem}[Estimate for the improved velocity] $ $\label{theorem:ufullpost}\\
Let (A1)--(A4) hold. 
Then for $0 \le n \le N$, there holds
\begin{align*}
\|\widetilde u_h^{\,n} - u(t^n)\|_{L^2(\Omega)} 
\le \tau^2 (C_1(u) + C_3(u)) + h^2 (C_1(u) + C_2(u)),
\end{align*}
with $C_1(u)$, $C_2(u)$ as defined in Theorems~\ref{theorem:discreteerror} and \ref{theorem:postu}, and $C_3(u)=C \|\dtt u\|_{L^\infty(0,T;H(\div;\Omega))}$.
\end{theorem}
\begin{proof}
We proceed with similar arguments as on the semi-discrete level. 
As a first step, we use a splitting of the error 
\begin{align*}
\|\widetilde u_h^{\,n} - u(t^n)\|_{L^2(\Omega)} 
 \le \|\widetilde u_h^{\,*}(t^n) - u(t^n)\|_{L^2(\Omega)} +  \|\widetilde u_h^{\,n} - \widetilde u_h^{\,*}(t^n)\|_{L^2(\Omega)} 
 = (i) + (ii)
\end{align*}
into an approximation error and a discrete error component. 
The first part can be readily estimated by Lemma~\ref{lemma:approximations2}, which yields
\begin{align*}
(i) = \|\widetilde u_h^{\,*}(t^n) - u(t^n)\|_{L^2(\Omega)}\le h^2 C_2(u).
\end{align*} 
Now observe that, by the variational characterizations of $\widetilde u_h^{\,n}$ and $\widetilde u_h^{\,*}(t^n)$, we have
\begin{align*}
(\widetilde u_h^{\,n}  -\widetilde u_h^{\,*}(t^n), v_h)_\Omega 
&= (\widehat u_h^{\,n},v_h)_{h,\Omega} + (\widetilde r_h^{\,n} , \div\,v_h)_\Omega - (\widetilde u_h^{\,*}(t^n), v_h)_\Omega \\
&= (\widehat u_h^{\,0},v_h)_{h,\Omega} + \sum_{k=1}^n \frac{\tau}{2}(\dtau u_h^{k} + \dtau u_h^{k-1}, v_h)_{h,\Omega} + 
(\widetilde r_h^{\,n}, \div\,v_h)_\Omega \\
& \qquad \qquad \qquad   - (\widetilde u_h^{\,*}(0), v_h)_\Omega - \int_0^{t^n} (\dt \widetilde u_h^{\,*}(s), v_h)_\Omega\, ds \\
&= (\widehat u_h^{\,0} - u_h^*(0), v_h)_{h,\Omega} + (\widetilde r_h^{\,n} + r_h^*(0) - \widetilde r_h^{\,*}(0), \div\,v_h)_\Omega \\
& \qquad \qquad \qquad + \frac{\tau}{2}\sum_{k=1}^n(p_h^k + p_h^{k-1}, \div\,v_h)_\Omega - 
\int_0^{t^n} (p_h^*(s), \div\,v_h)_\Omega\, ds.
\end{align*}
If $\div\,v_h=0$, then only the first term on the right hand side of this identity remains. 
From Problem~\ref{problem:fulldiscrete}, we further get
\begin{align*}
(u_h^{1/2},v_h)_{h,\Omega} = (u_h^{-1/2},v_h)_{h,\Omega} + \tau (p_h^0,\div\,v_h)_\Omega.
\end{align*}
Using the definition of $\widehat u_h^{\,0}$ and of $u_h^{-1/2}$ in \eqref{eq:sysvdf1}, 
we further compute
\begin{align*}
(\widehat u_h^{\,0}-u_h^*(0),v_h)_{h,\Omega} &= \frac{1}{2} (u_h^{1/2} + u_h^{-1/2},v_h)_{h,\Omega}- (u_h^*(0),v_h)_{h,\Omega} \\
&= (u_h^{-1/2},v_h)_{h,\Omega} + \frac{\tau}{2}(p_h^0,\div\,v_h)_\Omega - (u_h^0,v_h)_{h,\Omega}=0.
\end{align*}
As a consequence of these observations, we can see that 
\begin{align} \label{eq:aux}
(\widetilde u_h^{\,n}  -\widetilde u_h^{\,*}(t^n), v_h)_\Omega = 0 \qquad \text{if } \div \, v_h = 0. 
\end{align}
With similar arguments as on the semi-discrete level, 
we can now estimate the second term in the error splitting at the beginning of the proof by
\begin{align*}
|(ii)|^2 
&= \|\widetilde u_h^{\,n}  -\widetilde u_h^{\,*}(t^n)\|^2_{L^2(\Omega)} \\
&= (\widetilde u_h^{\,n}  -\widetilde u_h^{\,*}(t^n),\widetilde u_h^{\,n}  - 
\widetilde u_h^{\,*}(t^n) + u_h^*(t^n) - \widehat u_h^{\,n})_\Omega
 - (\widetilde u_h^{\,n}  -\widetilde u_h^{\,*}(t^n),u_h^*(t^n) - \widehat u_h^{\,n})_\Omega \\
&\le \|\widetilde u_h^{\,n}  -\widetilde u_h^{\,*}(t^n)\|_{L^2(\Omega)} \|u_h^*(t^n)-\widehat u_h^{\,n}\|_{L^2(\Omega)}.
\end{align*}
For the last step, we used that, by construction of the approximations, there holds
\begin{align*}
\div \, u_h^*(t^n) = \pi_h^0 \div\,u(t^n) = \div \, \widetilde u_h^{\,*}(t^n) 
\quad \text{and} \quad   
\div \, \widetilde u_h^{\,n} = \div\, \widehat u_h^{\,n}.
\end{align*}
Since the divergence of the test function is zero, we deduce from \eqref{eq:aux} that the first term in the second line vanishes.
We can thus proceed by 
\begin{align*}
|(ii)|&\le \|u_h^*(t^n)-\widehat u_h^{\,n}\|_{L^2(\Omega)} 
\le \|\widehat u_h^{\,*}(t^n) - \widehat u_h^{\,n}\|_{L^2(\Omega)} 
+ \|u_h^*(t^n) - \widehat u_h^{\,*}(t^n)\|_{L^2(\Omega)} 
\\
&= (iii) + (iv).
\end{align*}
The term (iii) in this estimate is covered by Theorem~\ref{theorem:fulldiscreteerror} and the remaining term (iv) can be bounded by 
\begin{align*}
|(iv)| 
\le \tau^2 \|\dtt u_h^*\|_{L^\infty(0,T;L^2(\Omega))} 
\le \tau^2 \|\dtt u\|_{L^\infty(0,T;H(\div;\Omega))}.
\end{align*}
In the last step, we used the stability of the inexact elliptic projection
in $H(\div;\Omega)$. The assertion of the theorem now follows by combination of all intermediate estimates.
\end{proof}

\section{Numerical tests} \label{sec:num}

We now illustrate our theoretical results by some numerical tests similar to those reported in \cite{LeeMadsen90,Monk92a}. 
It can be easily verified that for any wave vector $k=(k_1,k_2) \in \RR^2$ with $|k|=1$ and a function $g\in C^\infty(\RR)$, the plane wave functions given by
\begin{align}
u_{pw}(x,y,t) &= \binom{k_1}{k_2}  g(k_1x+k_2y-t), \label{eq:uex}\\
p_{pw}(x,y,t) &= g(k_1x+k_2y-t),  \label{eq:pex}
\end{align}
satisfy the first order wave equation \eqref{eq:sysm1}-\eqref{eq:sysm2} for all $(x,y) \in \RR^2$ and $t \in \RR$. 
Moreover, the solution $(u,p)$ is infinitely differentiable. 
For our numerical tests, we choose some domain $\Omega \subset \RR^2$ and use the explicit formulas for the pressure and the velocity to define Dirichlet boundary values for the pressure at $\partial\Omega_D$ and initial values for $u$ and $p$ at time $t=0$. The corresponding fully discrete finite element scheme with mass-lumping and leapfrog time stepping then reads 
\begin{alignat*}{2}
(\dtau u_h^{n},v_h)_{h,\Omega} - (p_h^{n}, \div \, v_h)_\Omega 
     &= -(p_{pw}(t^n),n\cdot v_h)_{\partial\Omega_D} \qquad &&\forall v_h \in V_h, \\
(\dtau p_h^{n+1/2},q_h)_\Omega + (\div \, u_h^{n+1/2},q_h)_\Omega 
     &= 0 \qquad &&\forall q_h \in Q_h.
\end{alignat*}
The inhomogeneous right hand side only requires some minor modifications in the definition of the auxiliary functions introduced in Problem~\ref{problem:auxfunctions} and Problem~\ref{problem:auxfunctions2}, 
as well as in the definition of the initial value $u_h^{-1/2}$ for the fully discrete scheme.

\subsection{Plane wave propagation}

We choose $g(x) = e^{-2(x+5)^2}$ and $k=\frac{1}{\sqrt{5}}(2,1)$ in the formulas \eqref{eq:uex}--\eqref{eq:pex} 
to define the analytical solution and consider acoustic wave propagation in the domain $\Omega = (-1,1)^2$ and
for the time interval $0 \le t \le T=5$. 
The initial values for $u$ and $p$ and the boundary values for $p$ are defined by the analytical formulas for the exact solution. 
Since the solution is infinitely differentiable here, we expect to observe the optimal convergence rates predicted by our theory. 
We use $\tnorm e\tnorm=\max_{0 \le t^n \le T} \|e(t^n)\|_{L^2(\Omega)}$ to denote the norm of the error. 
For our convergence study, we consider a sequence $\{\T_h\}_{h}$ of quasi-uniform but non-nested meshes $\T_h$ with decreasing mesh size $h=2^{-k}$, $k=1,2,\ldots$. 
In order to satisfy the CFL condition (A4), we choose $\tau = h/4$ as the time step size. 
The results of our computational tests are depicted in Table~\ref{tab:1}.

\begin{table}[ht!]
\begin{tabular}{c|c||c|c||c|c||c|c} 
$h$ & $\tau$ & $\tnorm \widehat u - \widehat u_h\tnorm$ & eoc & $\tnorm p-p_h\tnorm$ & eoc 
& $\tnorm \widehat u_h - \widehat u_h^{\,*}\tnorm$ & eoc \\
\hline
\hline
\rule{0pt}{2.1ex} 
$2^{-3}$ & $2^{-5}$ & $0.053047$ & ---    & $0.069893$ & ---    & $0.051699$ & ---    \\
$2^{-4}$ & $2^{-6}$ & $0.020622$ & $1.36$ & $0.033095$ & $1.08$ & $0.013408$ & $1.95$ \\
$2^{-5}$ & $2^{-7}$ & $0.009977$ & $1.05$ & $0.016408$ & $1.01$ & $0.003353$ & $2.00$ \\  
$2^{-6}$ & $2^{-8}$ & $0.004883$ & $1.03$ & $0.008114$ & $1.02$ & $0.000830$ & $2.01$
\end{tabular}
\medskip
\caption{Errors vs. mesh size $h$ and time step $\tau$ as well as the estimated order of convergence (eoc) for plane wave solution on a rectangular domain. \label{tab:1}} 
\end{table}

As predicted by Theorem~\ref{theorem:fulldiscreteerror2}, we obtain only first order convergence in both solution components, while the discrete error contribution for the velocity is of second order. 
In Table~\ref{tab:2}, we display the corresponding errors obtained after post-processing. 

\begin{table}[ht!]
\begin{tabular}{c|c||c|c||c|c} 
$h$ & $\tau$ & $\tnorm u - \widetilde u_h\tnorm$ & eoc & $\tnorm p - \widetilde p_h\tnorm$ & eoc \\
\hline
\hline
\rule{0pt}{2.1ex}
$2^{-3}$ & $2^{-5}$ & $0.051792$ & ---    & $0.055946$ & ---    \\
$2^{-4}$ & $2^{-6}$ & $0.013486$ & $1.94$ & $0.013180$ & $2.09$ \\
$2^{-5}$ & $2^{-7}$ & $0.003375$ & $2.00$ & $0.003220$ & $2.03$ \\
$2^{-6}$ & $2^{-8}$ & $0.000836$ & $2.01$ & $0.000791$ & $2.02$ \\  
\end{tabular}
\medskip
\caption{Errors vs. mesh size $h$ and time step $\tau$ as well as the estimated order of convergence (eoc) on a rectangular domain. \label{tab:2}} 
\end{table}

As expected from the results of Theorem~\ref{theorem:pfullpost} and Theorem~\ref{theorem:ufullpost}, 
we observe the full second order convergence for both solution components after post-processing. 
To highlight the qualitative improvement obtained by the post-processing, we display in Figure~\ref{fig:1} 
snapshots of the discrete solution components before and after post-processing. 
\begin{figure}[ht!]
\begin{center}
\includegraphics[width=.25\textwidth]{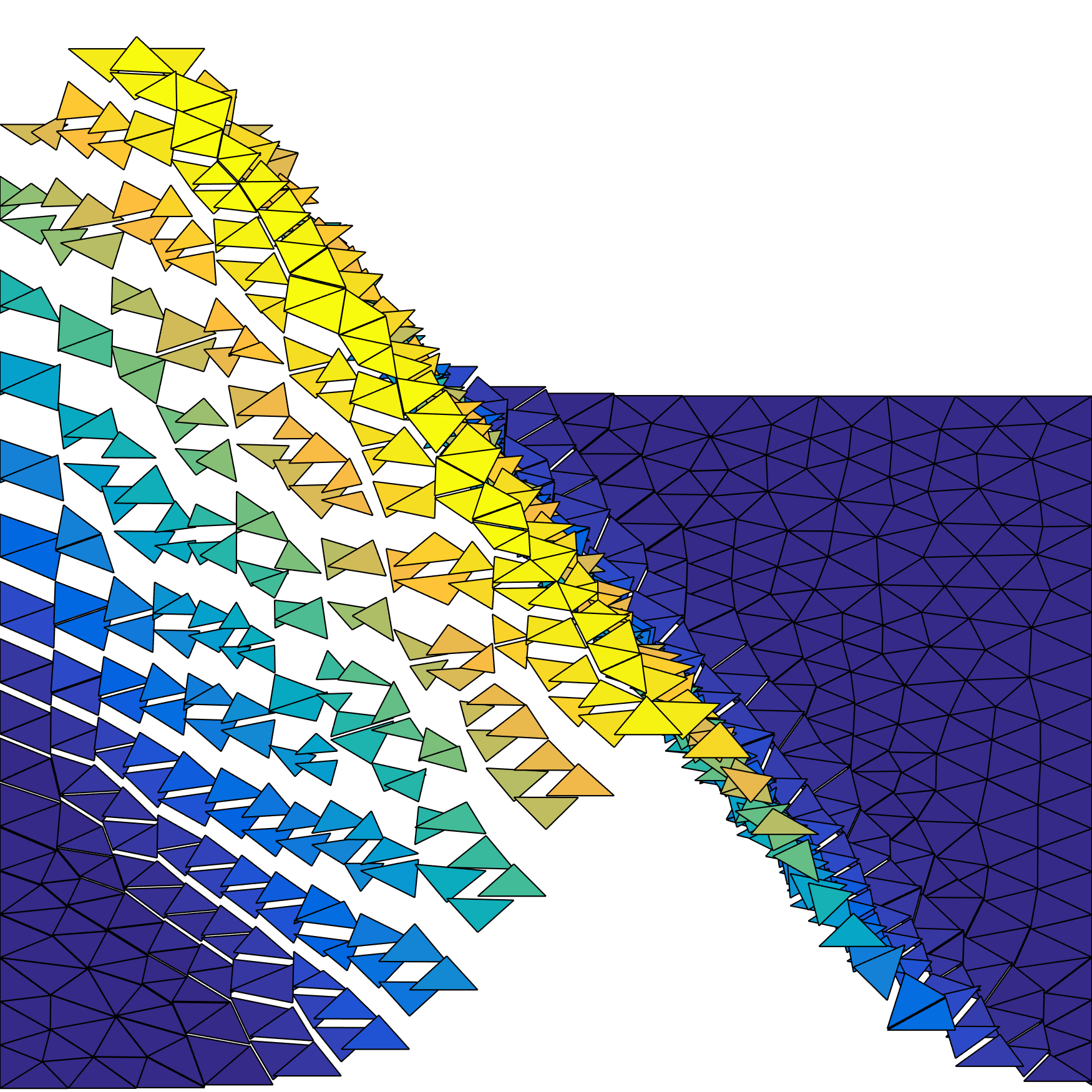}
\hskip3ex
\includegraphics[width=.25\textwidth]{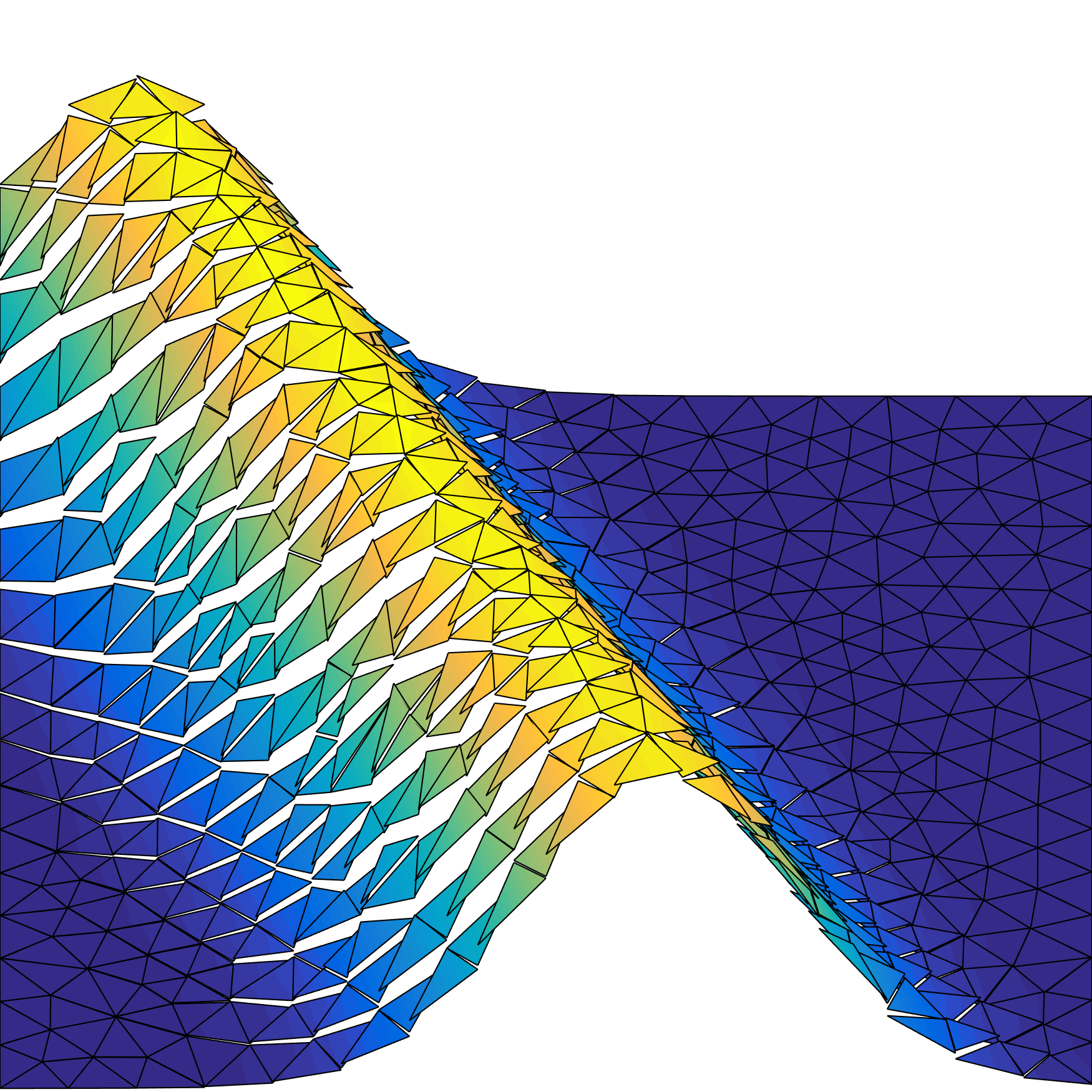}
\hskip3ex
\includegraphics[width=.25\textwidth]{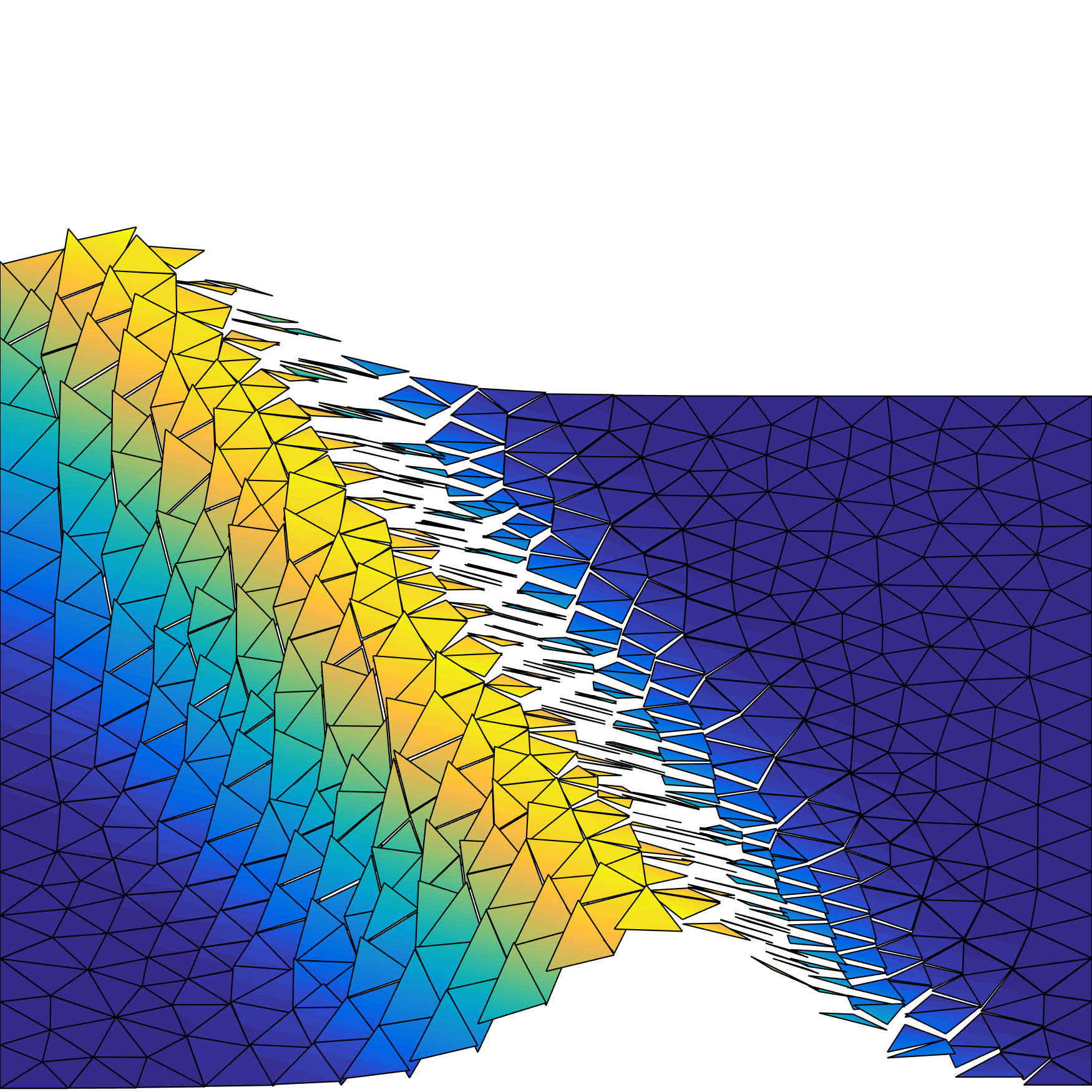} \\

\includegraphics[width=.25\textwidth]{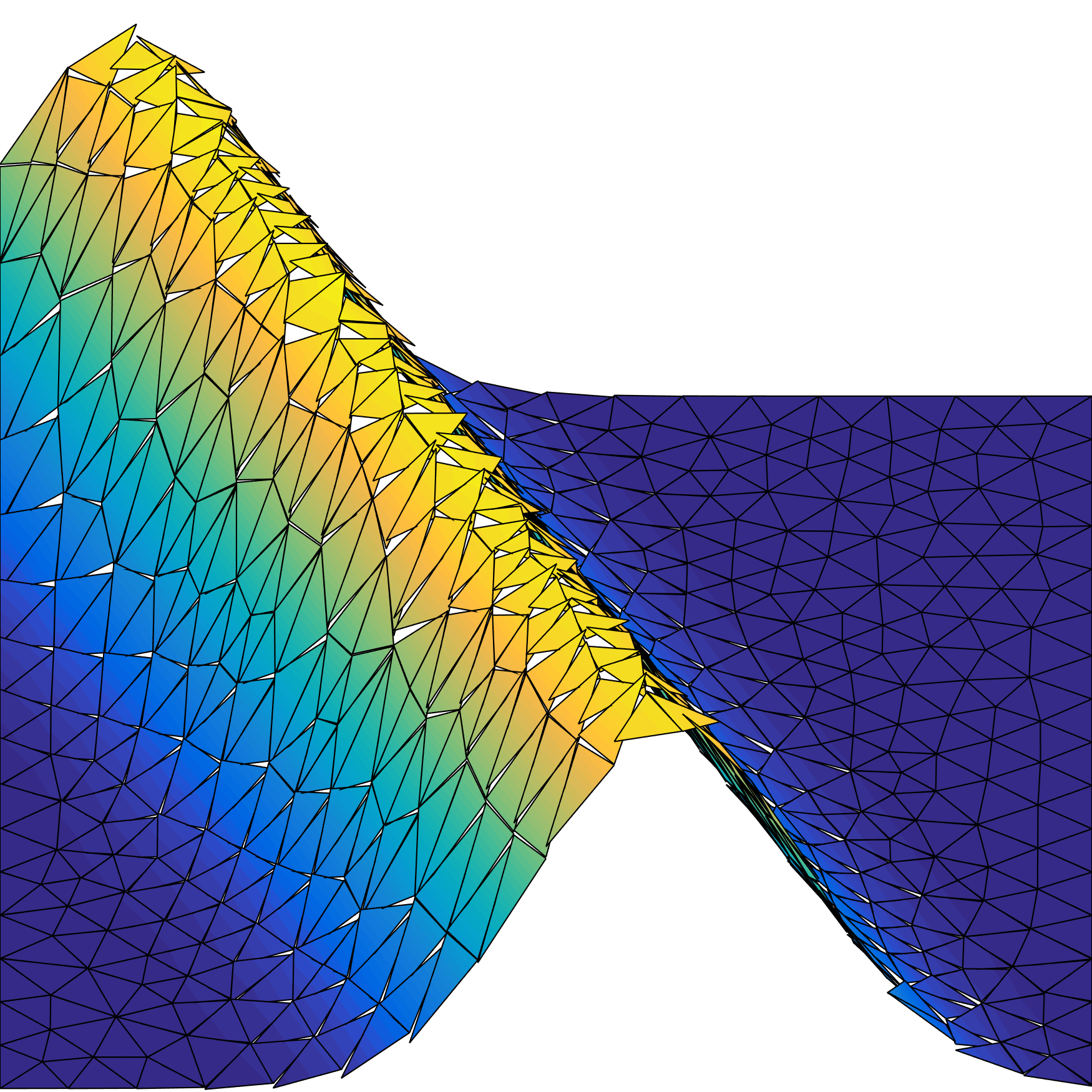}
\hskip3ex
\includegraphics[width=.25\textwidth]{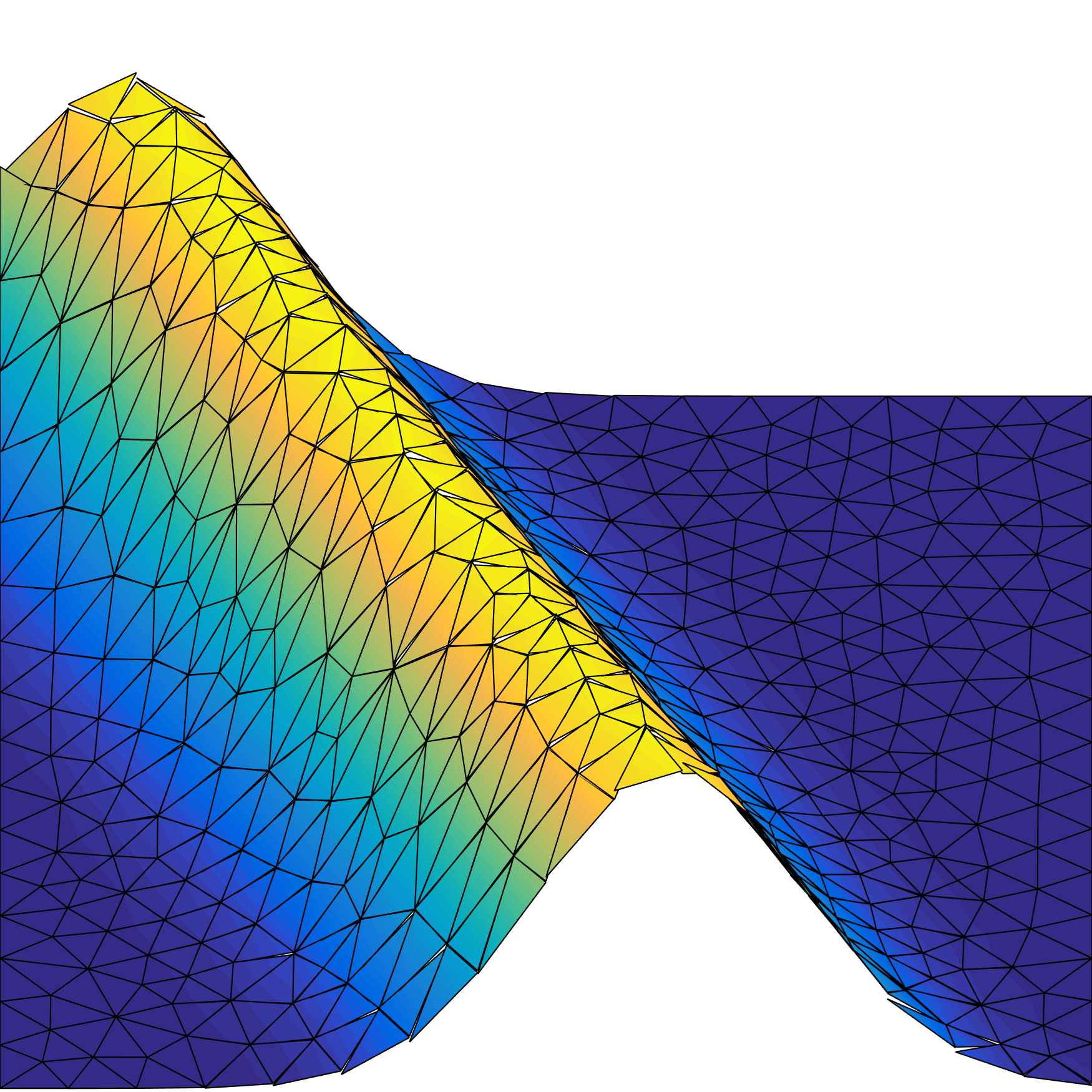}
\hskip3ex
\includegraphics[width=.25\textwidth]{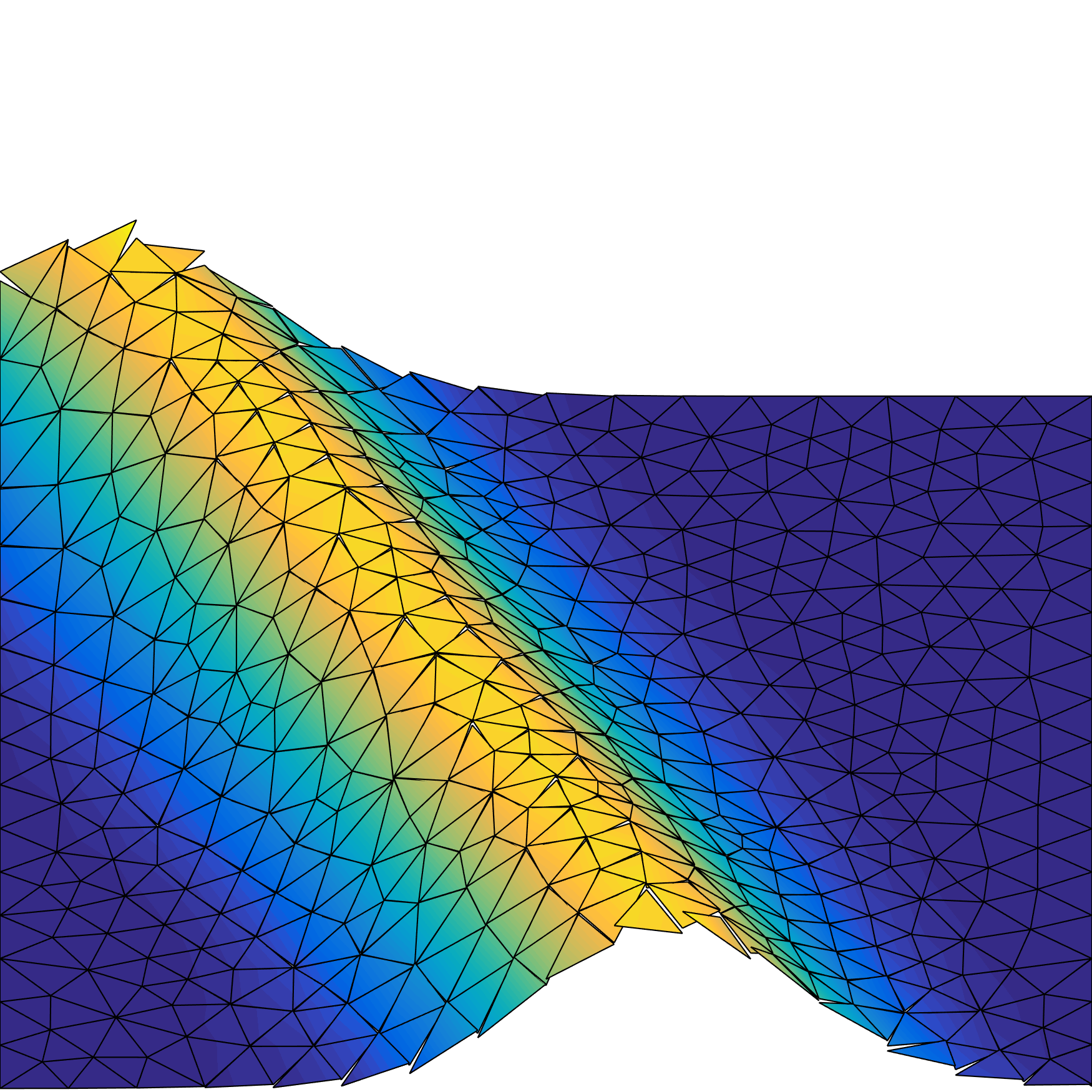}
\caption{Snapshots of the discrete approximations $p_h$, $u_{1,h}$, $u_{2,h}$ for the plane wave solution at time $t=2$ (top row) and corresponding approximations $\widetilde p_h$, $\widetilde u_{1,h}$, and $\widetilde u_{2,h}$ after post-processing (bottom row).}
\end{center}
\label{fig:1}
\end{figure}
As can be seen from the illustrations, the post-processing not only reduces the error quantitatively, but it also leads to almost continuous approximations. 

Let us remark that we observed second order convergence after post-processing also when repeating the same tests for the L-shape $\Omega=(-1,1)^2\setminus[0,1]^2$, although the convexity assumption (A2), which was needed for our analysis, was violated. 
This indicates that some of our assumptions might still be further relaxed.

\subsection{Scattering on a sphere}

As a second test case, we study the scattering of a plane wave by a cylinder.
Using symmetry in the third spatial direction, it suffices to consider again a two  dimensional 
test problem. 
The computational domain here is chosen as 
\begin{align*}
\Omega = (-1,1)^2\setminus \{(x,y)\,|\, \|(x,y+1)\|_{2} \le 0.2\};
\end{align*}
see Figure~\ref{fig:2} for a sketch. 
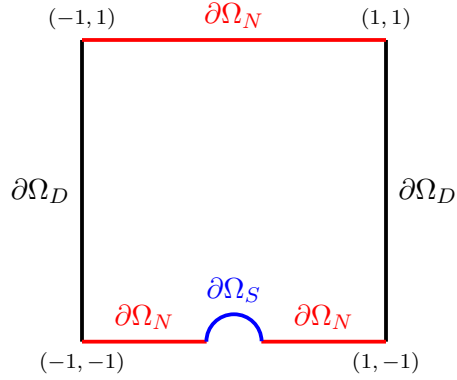
\begin{figure}[ht!]
\centering
\begin{tikzpicture}[scale=2]
\draw[black,line width=0.5mm] (-1,-1) -- (-1, 1) node[left,midway]  {\small $\partial\Omega_{D}$};
\draw[black,line width=0.5mm] ( 1,-1) -- ( 1, 1) node[right,midway] {\small $\partial\Omega_{D}$};
\draw[red, line width=0.5mm]  ( 1, 1) -- (-1, 1) node[above,midway] {\small $\partial\Omega_{N}$};
\draw[red, line width=0.5mm]  (-1,-1) -- ( -0.18,-1) node[above,midway] {\small $\partial\Omega_{N}$};
\draw[red, line width=0.5mm]  (0.18,-1) -- ( 1,-1)   node[above,midway] {\small $\partial\Omega_{N}$};
\draw[blue,line width=0.5mm,yshift=-1cm] (180:1ex) arc (180:1:1ex) node[above,midway] {\small $\partial\Omega_S$};;
\draw[] (-1,-1) node[below] {\tiny $(-1,-1)$};
\draw[] (1,-1) node[below] {\tiny $(1,-1)$};
\draw[] (-1,1) node[above] {\tiny $(-1,1)$};
\draw[] (1,1) node[above] {\tiny $(1,1)$};
\end{tikzpicture}
  \caption{Computational domain for the wave scattering problem.\label{fig:2}}
\end{figure}
We again consider the system \eqref{eq:sysm1}--\eqref{eq:sysm2} on this domain $\Omega$ for $0 \le t \le T=2$ 
with boundary conditions 
\begin{alignat*}{3}
p &= 0             && \text{ on }\partial\Omega_{S},\\
p &= p_{pw} \qquad  && \text{ on }\partial\Omega_{D},\\
n\cdot u &=0       && \text{ on }\partial\Omega_{N},
\end{alignat*}
where $p_{pw}$ is given by \eqref{eq:pex} with $g(x) = 2e^{-10(x+3)^2}$ 
and $k=(k_1,k_2)=(1,0)$. 
This test case models a plane wave that enters the computational domain from the left boundary 
and which is scattered at the circular boundary $\partial\Omega_S$. 
Some snapshots of the numerical solution after post-processing are depicted in Figure~\ref{fig:3}. 

\begin{figure}[ht!]
\captionsetup[subfloat]{width=120pt}
\begin{center}
\includegraphics[width=.25\textwidth]{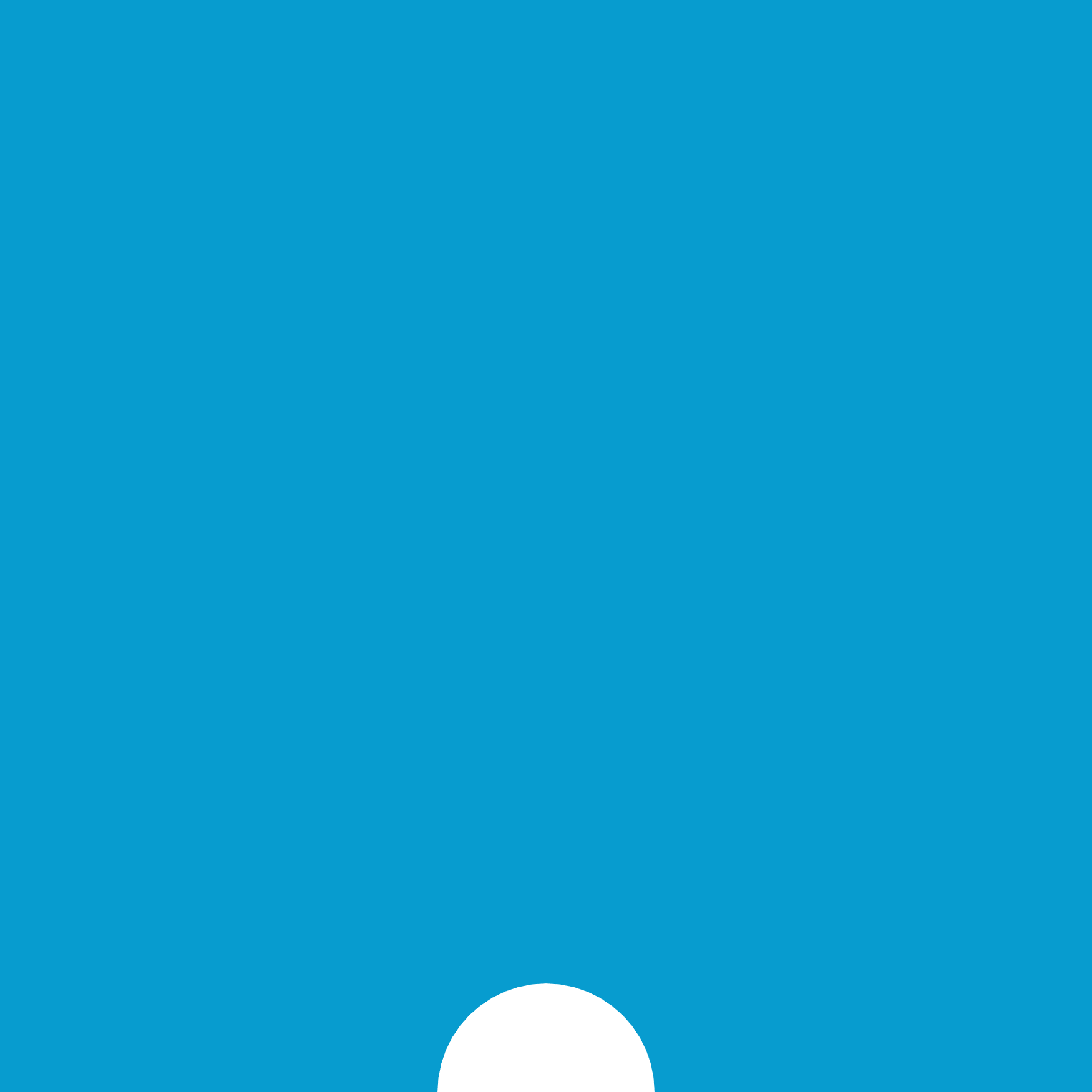}
\hskip3ex
\includegraphics[width=.25\textwidth]{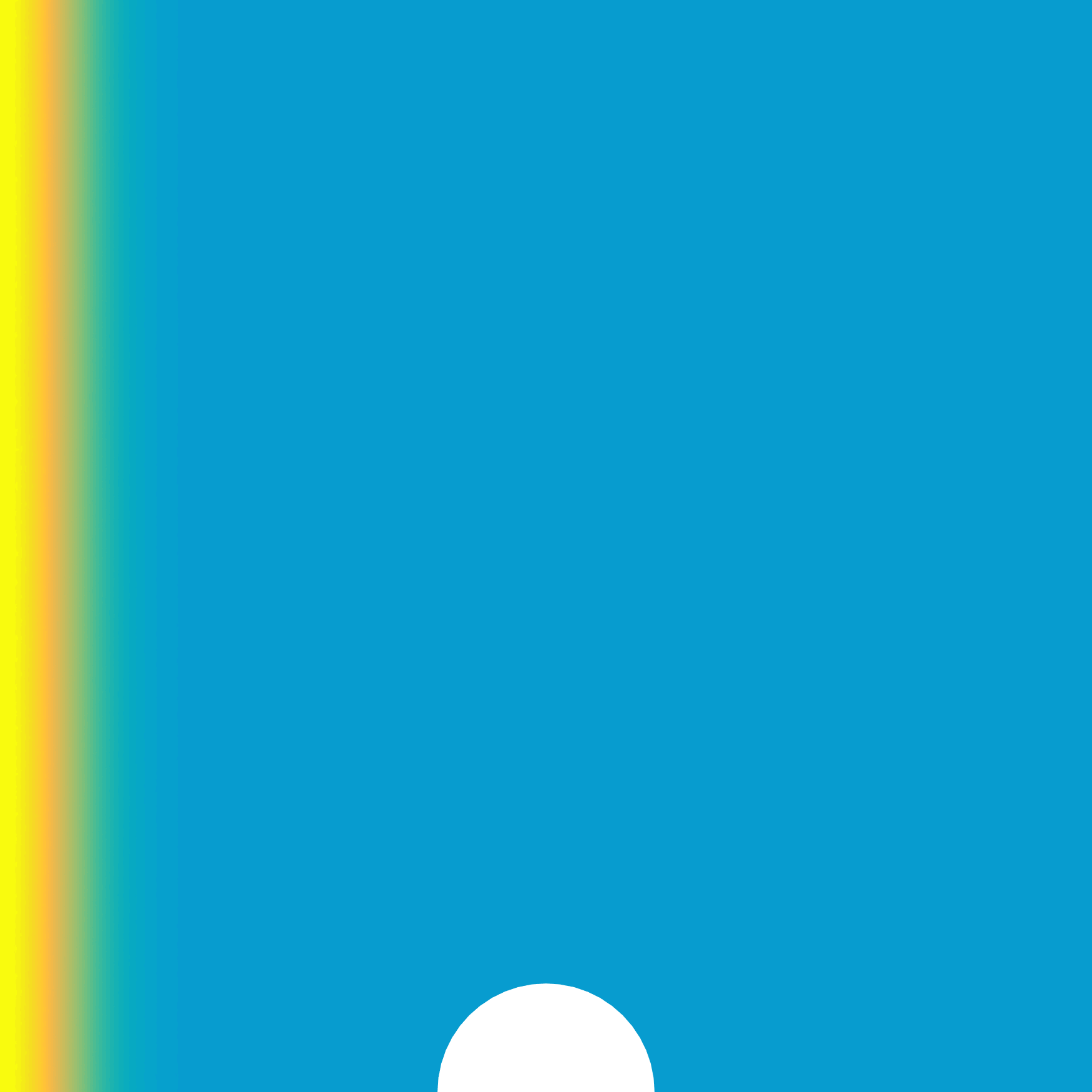}
\hskip3ex
\includegraphics[width=.25\textwidth]{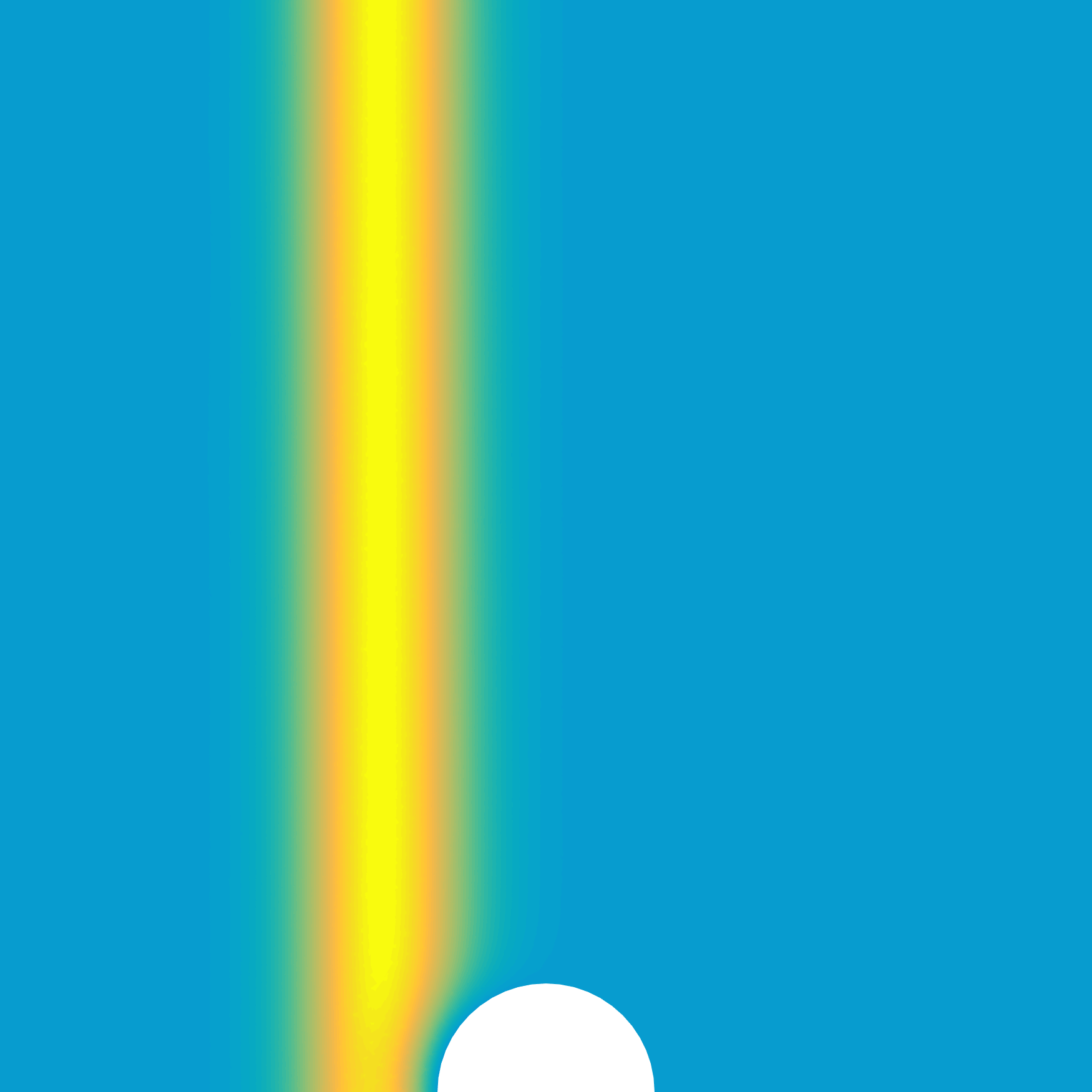} \\[1em]

\includegraphics[width=.25\textwidth]{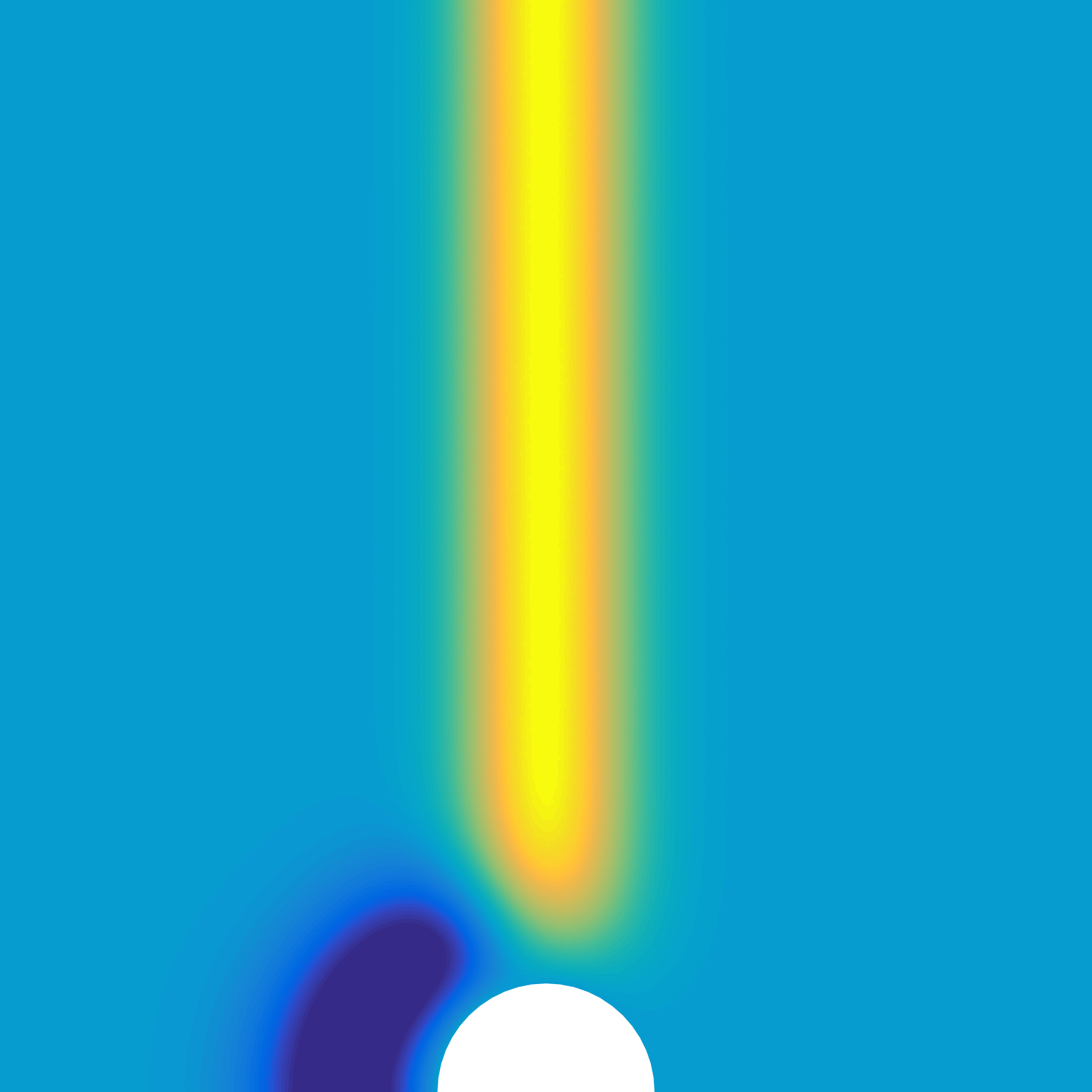}
\hskip3ex
\includegraphics[width=.25\textwidth]{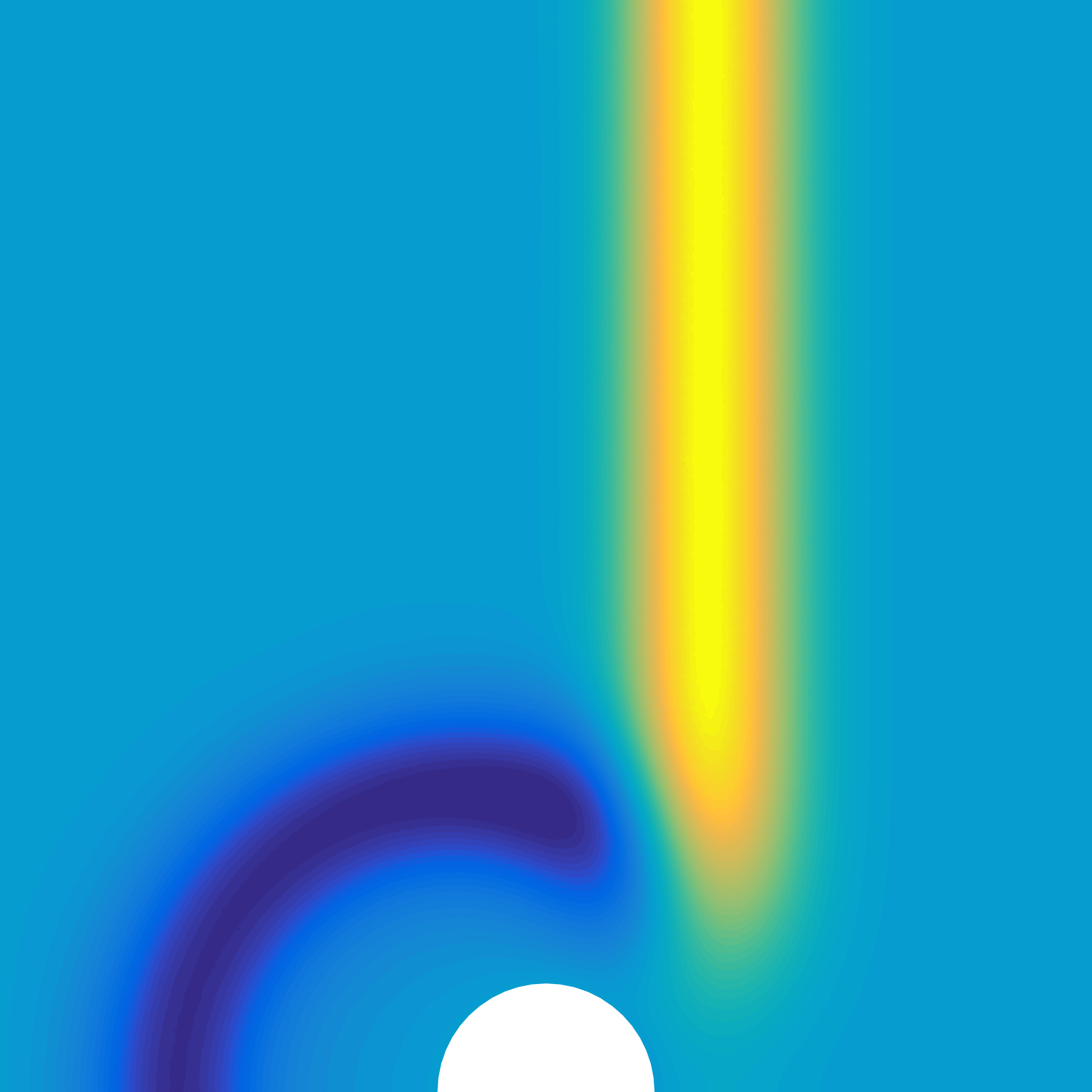}
\hskip3ex
\includegraphics[width=.25\textwidth]{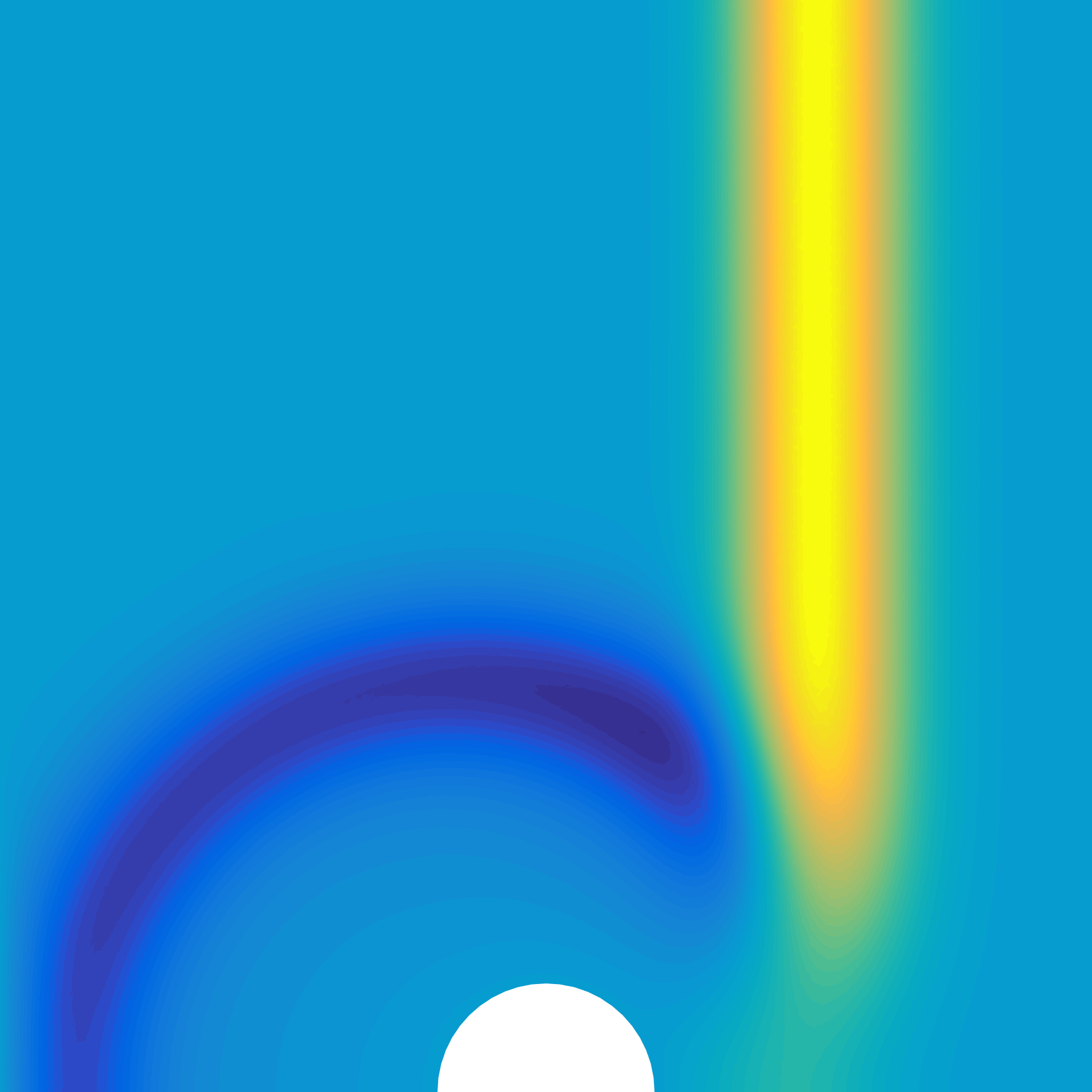} \\[1em]

\includegraphics[width=0.6\textwidth]{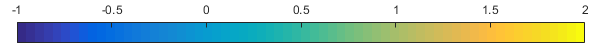}
\caption{Snapshots of the post-processed pressure fields $\widetilde p_h$ for time $t=0,0.5,1.2$ (top) and $t=1.5,1.8,2$ (bottom).}
\end{center}
\label{fig:3}
\end{figure}

To evaluate the convergence of the discretization scheme for this test problem, we utilize a sequence of meshes $\T_h$ obtained by uniform refinement of a quasi-uniform coarse mesh with $h=2^{-3}$. 
To guarantee a good approximation of the geometry, the vertices generated by refinement of edges at the curved boundary $\partial\Omega_S$ are projected to the exact circle after every refinement step. 
Note that $\T_{h/2}=F_{h/2}(\widetilde \T_{h/2})$ is a piecewise affine transformation of the mesh $\widetilde \T_{h/2}$ obtained by the usual regular refinement of the mesh $\T_h$. 
As a consequence, the meshes $\T_{h}$ are nested topologically but not geometrically. 
Any piecewise polynomial $p_h \in \P_k(\T_h)$ can however be prolongated to a piecewise polynomial $\pi_{h/2} p_h \in \P_k(\T_{h/2})$ defined by $\pi_{h/2} p_{h} = p_h \circ F_{h/2}^{-1}$; this allows us to compare discrete functions on different mesh levels.
In Table~\ref{tab:4}, we list the errors obtained in our computations after the post-processing was applied.
\begin{table}[h]
\begin{tabular}{c||c|c||c|c} 
$h$ & $\tnorm \widetilde u_{h/2} - \pi_{h/2}\widetilde u_h\tnorm$ & eoc & $\tnorm \widetilde p_{h/2} - \widetilde \pi_{h/2} p_h\tnorm$ & eoc \\
\hline
\hline
\rule{0pt}{2.1ex}
$2^{-3}$  & $0.359097$ & ---    &  $0.410665$ & ---    \\
$2^{-4}$  & $0.094968$ & $1.92$ &  $0.099264$ & $2.05$ \\
$2^{-5}$  & $0.023241$ & $2.03$ &  $0.023688$ & $2.07$ \\
$2^{-6}$  & $0.005753$ & $2.01$ &  $0.005829$ & $2.02$
\end{tabular}
\medskip
\caption{Errors and estimated order of convergence (eoc) for post-processed solution obtained with time step size $\tau=1/1000$.\label{tab:4}} 
\end{table}
Let us note that, although the convexity assumption (A2) is violated and despite the inexact representation of the geometry, we still observe convergence of second order after post-processing. These computational results indicate that 
our analysis might be extended in several ways.

\section{Discussion}

In this paper we considered the numerical approximation of acoustic wave propagation formulated as a first order hyperbolic system by a mixed finite element method with $\BDM_1$-$\P_0$ elements. An appropriate mass-lumping strategy was utilized which leads to block-diagonal mass matrices and allows an efficient time integration by explicit Runge-Kutta and multistep methods. 
Due to the perturbations introduced by the mass-lumping, the consistency error of this method is only of first order. 
Nevertheless, the numerical approximation carries second order information which was used to construct piecewise linear second order approximations for both solution components by certain post-processing procedures.
The resulting scheme can be interpreted as a generalization of finite-difference time-domain methods to unstructured grids. 

The theoretical results were illustrated by computational tests which indicate that some of the assumptions needed for our analysis, e.g., the convexity of the domain under consideration, might still be further relaxed. In addition, we also observed second order convergence for piecewise linear approximations of curved domains which is not covered by our theory yet. 
Let us finally note that our approach also seems applicable to problems in elasticity and electromagnetics. In two dimension, the latter can be obtained by simple rotation of the basis functions. Further extensions will be discussed elsewhere. 


\section*{Acknowledgements}

The authors are grateful for financial support by the ``Excellence Initiative'' of the German Federal and State Governments via the Graduate School of Computational Engineering GSC~233 at Technische Universität Darmstadt and by the German Research Foundation (DFG) via grants IRTG~1529 and TRR~154 project C4. 


\begin{thebibliography}{10}

\bibitem{Baker76}
G.~A. Baker.
\newblock Error estimates for finite element approximations for second order
  hyperbolic equations.
\newblock {\em SIAM J. Numer. Anal.}, 13:564--576, 1976.

\bibitem{BakerBramble79}
G.~A. Baker and J.~H. Bramble.
\newblock Semidiscrete and single step fully discrete approximations for second
  order hyperbolic equations.
\newblock {\em RAIRO Anal. Numer.}, 13:75--100, 1979.

\bibitem{BoffiEtAl08}
D.~Boffi, F.~Brezzi, L.~F. Demkowicz, R.~G. Dur{\'a}n, R.~S. Falk, and
  M.~Fortin.
\newblock {\em Mixed finite elements, compatibility conditions, and
  applications}, volume 1939 of {\em Lecture Notes in Mathematics}.
\newblock Springer-Verlag, Berlin; Fondazione C.I.M.E., Florence, 2008.

\bibitem{BoffiBrezziFortin13}
D.~Boffi, F.~Brezzi, and M.~Fortin.
\newblock {\em Mixed finite element methods and applications}, volume~44 of
  {\em Springer Series in Computational Mathematics}.
\newblock Springer, Heidelberg, 2013.

\bibitem{BrezziDouglasMarini85}
F.~Brezzi, J.~Douglas, and L.~D. Marini.
\newblock Two families of mixed elements for second order elliptic problems.
\newblock {\em Numer. Math.}, 88:217--235, 1985.

\bibitem{Cohen02}
G.~Cohen.
\newblock {\em Higher-Order Numerical Methods for Transient Wave Equations}.
\newblock Springer, Heidelberg, 2002.

\bibitem{CohenJolyRobertsTordjman01}
G.~Cohen, P.~Joly, J.~E. Roberts, and N.~Tordjman.
\newblock Higher order triangular finite elements with mass lumping for the
  wave equation.
\newblock {\em SIAM J. Numer. Anal.}, 38(6):2047--2078, 2001.

\bibitem{CohenMonk98}
G.~Cohen and P.~Monk.
\newblock Gauss point mass lumping schemes for {M}axwell's equations.
\newblock {\em Numer. Meth. Part. Diff. Equat.}, 14:63--88, 1998.

\bibitem{CowsarDupontWheeler96}
L.~C. Cowsar, T.~F. Dupont, and M.~F. Wheeler.
\newblock A priori estimates for mixed finite element approximations of
  second-order hyperbolic equations with absorbing boundary conditions.
\newblock {\em SIAM J. Numer. Anal.}, 33:492--504, 1996.

\bibitem{DouglasDupontWheeler78}
J.~Douglas, T.~Dupont, and M.~F. Wheeler.
\newblock A quasi-projection analysis of {G}alerkin methods for parabolic and
  hyperbolic equations.
\newblock {\em Math. Comp.}, 32:345--362, 1978.

\bibitem{Dupont73}
T.~Dupont.
\newblock $l^2$ estimates for {G}alerkin methods for second-order hyperbolic
  equations.
\newblock {\em SIAM J. Numer. Anal.}, 10:880--889, 1973.

\bibitem{EggerRadu16}
H.~Egger and B.~Radu.
\newblock Super-convergence and post-processing for mixed finite element
  approximations of the wave equation.
\newblock Technical report, 2016.
\newblock arXiv:1608.03818.

\bibitem{ErnGuermond}
A.~Ern and J.-L. Guermond.
\newblock {\em Theory and practice of finite elements}, volume 159 of {\em
  Applied Mathematical Sciences}.
\newblock Springer-Verlag, New York, 2004.

\bibitem{Evans98}
L.~C. Evans.
\newblock {\em Partial Differential Equations}, volume~19 of {\em Graduate
  Studies in Mathematics}.
\newblock American Mathematical Society, 1998.

\bibitem{Geveci88}
T.~Geveci.
\newblock On the application of mixed finite element methods to the wave
  equations.
\newblock {\em RAIRO Model. Math. Anal. Numer.}, 22:243--250, 1988.

\bibitem{HairerLubichWanner03}
E.~Hairer, C.~Lubich, and G.~Wanner.
\newblock Geometric numerical integration illustrated by the
  {S}t\"ormer-{V}erlet method.
\newblock {\em Acta Numer.}, 12:399--450, 2003.

\bibitem{JenkinsRiviereWheeler02}
E.~W. Jenkins, T.~Rivi{\`e}re, and M.~F. Wheeler.
\newblock A priori error estimates for mixed finite element approximations of
  the acoustic wave equation.
\newblock {\em SIAM J. Numer. Anal.}, 40:1698--1715, 2002.

\bibitem{Joly03}
P.~Joly.
\newblock Variational methods for time-dependent wave propagation problems.
\newblock In {\em Topics in Computational Wave Propagation}, volume~31 of {\em
  LNCSE}, pages 201--264. Springer, 2003.

\bibitem{LeeMadsen90}
R.~L. Lee and N.~K. Madsen.
\newblock A mixed finite element formulation for {M}axwell's equations in the
  time domain.
\newblock {\em J. Comput. Phys.}, 88(2):284--304, 1990.

\bibitem{Makridakis92}
C.~G. Makridakis.
\newblock On mixed finite element methods for linear elastodynamics.
\newblock {\em Numer. Math.}, 61:235--260, 1992.

\bibitem{MakridakisMonk95}
C.~G. Makridakis and P.~Monk.
\newblock Time-discrete finite element schemes for {M}axwell's equations.
\newblock {\em RAIRO Model. Math. Anal. Numer.}, 29:171--197, 1995.

\bibitem{Monk92}
P.~Monk.
\newblock Analysis of a finite element methods for {M}axwell's equations.
\newblock {\em SIAM J. Numer. Anal.}, 29:714--729, 1992.

\bibitem{Monk92a}
P.~Monk.
\newblock A comparison of three mixed methods for the time-dependent
  {M}axwell's equations.
\newblock {\em SIAM J. Sci. Statist. Comput.}, 13:1097--1122, 1992.

\bibitem{MonkSuli94}
P.~Monk and E.~S\"uli.
\newblock A convergence analysis of {Y}ee's scheme on nonuniform grids.
\newblock {\em SIAM J. Numer. Anal.}, 31(2):393--412, 1994.

\bibitem{Pazy83}
A.~Pazy.
\newblock {\em Semigroups of linear operators and applications to partial
  differential equations}, volume~44 of {\em Applied Mathematical Sciences}.
\newblock Springer-Verlag, New York, 1983.

\bibitem{Stenberg91}
R.~Stenberg.
\newblock Postprocessing schemes for some mixed finite elements.
\newblock {\em RAIRO Model. Math. Anal. Numer.}, 25:151--167, 1991.

\bibitem{Taflove}
A.~Taflove and S.~C. Hagness.
\newblock {\em Computational Electrodynamics: The Finite-Difference Time-Domain
  Method, 3rd ed.}
\newblock Artech House, Norwood, MA, 2005.

\bibitem{Wheeler73a}
M.~F. Wheeler.
\newblock {$L_{\infty }$} estimates of optimal orders for {G}alerkin methods
  for one-dimensional second order parabolic and hyperbolic equations.
\newblock {\em SIAM J. Numer. Anal.}, 10:908--913, 1973.

\bibitem{WheelerYotov06}
M.~F. Wheeler and I.~Yotov.
\newblock A multipoint flux mixed finite element method.
\newblock {\em SIAM J. Numer. Anal.}, 44(5):2082--2106, 2006.

\bibitem{Yee66}
K.~Yee.
\newblock Numerical solution of initial boundary value problems involving
  {M}axwell’s equations in isotropic media.
\newblock {\em IEEE Trans. Antennas and Propagation}, AP-16:302--307, 1966.

\end{thebibliography}

\end{document}